\documentclass[pdflatex,sn-mathphys-ay]{sn-jnl}

\usepackage{graphicx}%
\usepackage{multirow}%
\usepackage{amsmath,amssymb,amsfonts}%
\usepackage{amsthm}%
\usepackage{mathrsfs}%
\usepackage[title]{appendix}%
\usepackage{xcolor}%
\usepackage{textcomp}%
\usepackage{manyfoot}%
\usepackage{booktabs}%
\usepackage{algorithm}%
\usepackage{algorithmicx}%
\usepackage{algpseudocode}%
\usepackage{listings}%
\newtheorem{specialassumption}{}

\newtheorem{assumptions}{}

\usepackage{dsfont}
\usepackage{lscape}

\theoremstyle{thmstyleone}%
\newtheorem{theorem}{Theorem}
\theoremstyle{thmstyletwo}%
\newtheorem{example}{Example}%
\newtheorem{remark}{Remark}%
\newtheorem{corollary}{Corollary}%
\newtheorem{lemma}{Lemma}%

\theoremstyle{thmstylethree}%
\newtheorem{definition}{Definition}%

\begin{document}

\title[Article Title]{Model-assisted estimation of the finite population mean via local linear regression with associated kernels}


\author*[1]{\fnm{Sobom M.} \sur{Somé}}\email{sobom.some@uts.bf}

\author[1]{\fnm{Idrissa} \sur{Dimzouré}}\email{dimdimbao@gmail.com}


\affil[1]{\orgdiv{Laboratoire Sciences et Technologie}, \orgname{Université Thomas SANKARA }, \orgaddress{\street{Ouagadougou 12}, \city{Ouagadougou}, \postcode{12 BP 417}, \state{Kadiogo}, \country{Burkina Faso}}}


\abstract{In this paper, we propose a general regression-assisted estimator of the population mean based on local regression with associated kernels, a broad class that encompasses both classical (symmetric) kernels, such as the Gaussian and Epanechnikov kernels, and asymmetric kernels, such as the gamma kernel. We establish the asymptotic properties of the proposed estimator, including explicit expressions for the bias, mean squared error, and asymptotic normality, under a unified framework that applies to any kernel in this class. Through simulation studies on nonnegative data and one real-data application, we compare the finite-sample performance of the estimator based on the gamma kernel with those relying on the Gaussian and Epanechnikov kernels. When the data are supported on the positive real line, symmetric kernels suffer from boundary bias near zero due to mass assignment outside the support, whereas the gamma kernel, naturally adapted to $[0,\infty)$, alleviates this effect. The results confirm that the choice of kernel within the associated matters in practice, and that asymmetric kernels provide a reliable and efficient alternative when the underlying regression function is supported on the positive real line or exhibits marked asymmetry.}




\maketitle

\section{Introduction}\label{sec:Intro}	

Survey estimation methods have evolved considerably over the past decades with the aim of improving the accuracy and reliability of population parameter estimates such as totals or means. The well-known Horvitz--Thompson (HT) estimator \citep{horvitz1952generalization} provides a design-unbiased estimate of the population mean. However, its efficiency can often be improved by incorporating auxiliary information and modeling the relationship between the survey variable and auxiliary variables. In this context, the model-assisted approach \citep{sarnndal1980, robinson1983asymptotic, sarnndalwright1984, sarnndal1992model} has proven particularly effective by combining information from the sampling design with a working model of the relationship between the variable of interest and auxiliary variables. This framework has led to the development of increasingly efficient estimators \citep{breidt2000local, breidt2017model}.

A number of estimators have been developed within the model-assisted framework. The generalized regression estimator (GREG) \citep{cassel1977foundations} is among the most widely used. Other important developments include calibration estimators \citep{deville1992calibration}, empirical likelihood approaches \citep{chen1993empirical}, and several nonparametric regression-based estimators, including local regression \citep{breidt2000local, sanchez2014nonparametric}, spline-based methods \citep{breidt2005model, goga2014efficient}, additive models \citep{wang2011nonparametric}, and other regression approaches \citep{fuller2002regression, opsomer2005selecting, breidt2017model, liu2018geographically, mostafa2019finite}. More recently, modern statistical learning and Bayesian ideas have been incorporated into survey estimation. \citet{pohjankukka2024model} proposed a Bayesian framework for model-assisted survey modeling, while \citet{dagdoug2023model} employed random forests within an assisted estimation setting, and \citet{kwon2025debiased} introduced debiased calibration techniques based on generalized entropy.

Most existing nonparametric regression estimators in the survey sampling literature rely on classical symmetric kernels such as the Gaussian or Epanechnikov kernels. While these methods perform well when the auxiliary variable is approximately symmetrically distributed and supported on the real line , their performance can deteriorate when the support is bounded or restricted to the positive real line. In such settings, symmetric kernels assign positive mass outside the support of the data near the boundary, introducing a boundary bias that can substantially distort estimation accuracy, see for examples in density regression. A natural and principled remedy is the use of associated kernels \citep{somes2024}, a broad class that encompasses both classical symmetric kernels and asymmetric kernels, and whose members adapt their shape and support to the location of the estimation point \citet{stats4010013}. \citet{chen1999beta, chen2000beta} and \citet{chen2002local} were among the first to investigate asymmetric kernels for respectively probability density function estimation (pdf) and local linear regression on bounded supports. \citet{fe2014estimation} and \citet{genest2025local} highlighted the local in regression gamma dirichlet ability of these kernels to reduce boundary bias effectively on the support. \citet{stats4010013} provided further developments in density estimation using associated kernels, and \citet{Yue2026} demonstrated that asymmetric kernels can improve estimation accuracy by better accommodating the support structure of the data.

The objective of this paper is to propose a general model-assisted estimator of the finite population mean based on local linear regression with associated kernels. The asymptotic properties of the proposed estimator, including bias, mean squared error, and asymptotic normality, are derived theoretically under a unified framework that applies to any kernel in the associated class. This framework encompasses classical symmetric kernels such as those used in \citet{breidt2000local} as special cases, while also accommodating asymmetric kernels that are better adapted to bounded or positive supports. In particular, we investigate the use of the gamma kernel \citep{chen2002local} in local regression, which is naturally supported on $[0, \infty)$ and possesses attractive smoothing properties for positive and asymmetric data. We also defined bias and variance of the regression functions. The performance of the proposed estimator is evaluated through simulation studies and one real-data application, where comparisons are conducted between the gamma-kernel-based model assisted estimator and estimators based on the Epanechnikov and Gaussian kernels, as well as the HT estimator. The comparative study illustrates the practical consequences of kernel choice within the associated class, and shows that asymmetric kernels such as the gamma kernel offer efficiency gains when the support of the auxiliary variable departs from the real line.

The remainder of the paper is organized as follows. Section~\ref{sec:EL1} presents the proposed estimator and derives its asymptotic properties under the general associated kernel framework. Section~\ref{sec:ELLL} reports simulation studies and real-data application comparing the proposed approach with existing estimators under simple random sampling design without replacement (SRSWOR). Section~\ref{sec:Conclusion} concludes with a discussion and directions for future research, while technical proofs are collected in the Appendix.

\section{Model-assisted estimation using local linear regression with associated kernels}\label{sec:EL1}

Consider a finite population $U = \{1, \ldots, i, \ldots, N\}$ of size $N$ from which a sample $s$ of size $n$ is drawn according to a fixed-size sampling design. Let $\mathbb{P}(s)$ denote the probability of selecting sample $s$, $\pi_i = \mathbb{P}(i \in s)$ the first-order inclusion probability of unit $i$, and $\pi_{ij} = \mathbb{P}(i, j \in s)$ the joint inclusion probability of units $i$ and $j$. The goal is to estimate the finite population mean $ \overline{y} = N^{-1}\sum_{i \in U} y_i$. Let $\mathds{1}_i$ denote the sample membership indicator of unit $i$, equal to 1 if $i \in s$ and 0 otherwise, so that $\mathbb{E}_p(\mathds{1}_i) = \pi_i$, where $\mathbb{E}_p(\cdot)$ denotes expectation with respect to the sampling design. The HT estimator \citep{horvitz1952generalization} is given by
\begin{equation}
	\widehat{\overline{y}}_{\mathrm{HT}} = \frac{1}{N}\sum_{i \in s}^{}\frac{y_i}{\pi_i} = \frac{1}{N}\sum_{i \in U}\frac{y_i}{\pi_i}\mathds{1}_i.
	\label{bar_y_HT}
\end{equation}

To improve the efficiency of \eqref{bar_y_HT}, we assume that the study variable $y$ and the auxiliary variable $x$ are related through the superpopulation model
\begin{equation}
	\xi\ :\ y_i = m(x_i) + \epsilon_i, \quad i \in U,
	\label{xi}
\end{equation}
where $m(x_i) = \mathbb{E}_{\xi}(y_i \mid x_i)$ is the unknown regression function, $(x_i, y_i)$ are the observed values of $(x, y)$ for unit $i$, and the $\epsilon_i$ are independent random variables satisfying $\mathbb{E}_\xi(\epsilon_i \mid x_i) = 0$ and $\mathrm{Var}_\xi(\epsilon_i \mid x_i) = \sigma^2(x_i)$. Since $m(\cdot)$ is unknown, we estimate it nonparametrically using local linear regression with associated kernels, as described in the following subsections.

\subsection{Method}\label{subsec:method}

To estimate the finite population mean, we use local linear regression with associated kernels in the survey sampling framework. We begin by recalling the definition of associated kernels.

\begin{definition}[\citealp{esstafa2026properties}]
Let $\mathbb{T} \subseteq \mathbb{R}$ be the support of the pdf $f$ to be estimated, $x_j \in \mathbb{T}$ a target point, and $h>0$ a bandwidth. A parametrized pdf $K_{x_j,h}(\cdot)$ on the support $\mathbb{S}_{x_j,h} \subseteq \mathbb{R}$ is called a \textit{continuous associated kernel} if the following conditions hold:
	\begin{equation}
		x_j \in \mathbb{S}_{x_j,h}, \quad
		\mathbb{E}(\mathcal{Z}_{x_j,h}) - x_j =: A(x_j,h) \xrightarrow[h \to 0]{} 0
		\quad \text{and} \quad
		\operatorname{Var}(\mathcal{Z}_{x_j,h}) =: B(x_j,h) \xrightarrow[h \to 0]{} 0,
		\label{definition:associated_kernel}
	\end{equation}
	where $\mathcal{Z}_{x_j,h}$ denotes the continuous random variable with pdf $K_{x_j,h}(\cdot)$.
\end{definition}

To construct the local linear estimator, we introduce the following matrices defined over the finite population $U$:
\[
Y_U = \left[y_i\right]_{i \in U}, \quad
\mathcal{X}_U = \begin{bmatrix} 1 & x_i - x_j \end{bmatrix}_{i \in U}, \quad
\mathcal{W}_U = \mathrm{diag}\left\{K_{x_j, h}(x_i)\right\}_{i \in U}.
\]
The local linear estimator of $m(x_j)$ based on the entire population, centered at $x_j$, is given by
\begin{align}
	m_j = e_1^\top \left(\mathcal{X}_U^\top \mathcal{W}_U \mathcal{X}_U\right)^{-1} \mathcal{X}_U^\top \mathcal{W}_U Y_U,
	\label{m_j}
\end{align}
where $e_1 = (1, 0)^\top$. Since $m(\cdot)$ is unknown in practice and the full population is not observed, $m_j$ serves as the theoretical benchmark from which the sample-based estimator is derived.

A design-unbiased estimator of the population mean can be obtained through the generalized difference estimator \citep{cassel1977foundations, sarnndal1992model}, which reduces variance by exploiting the residuals $y_j - m_j$ rather than the raw values $y_j$:
\begin{equation}
	\overline{y}^{\ast} = \frac{1}{N}\sum_{j \in s}\frac{y_j - m_j}{\pi_j} + \frac{1}{N}\sum_{j \in U} m_j.
	\label{y*}
\end{equation}
The design variance of \eqref{y*} is
\begin{equation}
	\mathrm{Var}_p\left(\overline{y}^{\ast}\right) = \frac{1}{N^2}\sum_{i \in U}^{}\sum_{j \in U}^{}\left(\pi_{ij} - \pi_i \pi_j\right)\frac{y_i - m_i}{\pi_i}\frac{y_j - m_j}{\pi_j}.
	\label{var_y_bar*}
\end{equation}

In practice, $m_j$ is replaced by its sample-based estimate $\widehat{m}_j = \widehat{\beta}_0$, where $\widehat{\beta} = (\widehat{\beta}_0, \widehat{\beta}_1)^\top$ is obtained by solving the weighted least squares problem
\begin{equation}
	\arg\min_{\beta} \sum_{i=1}^{n} \frac{1}{\pi_i}\left\{y_i - \beta_0 - \beta_1(x_i - x_j)\right\}^2 K_{x_j, h}(x_i).
	\label{eq:WLS}
\end{equation}
The inclusion probability weights $1/\pi_i$ account for the unequal selection probabilities of the sampling design. Defining the sample-level matrices
\[
Y_s = \left[y_i\right]_{i \in s}, \quad
\mathcal{X}_s = \begin{bmatrix} 1 & x_i - x_j \end{bmatrix}_{i \in s}, \quad
\mathcal{W}_s = \mathrm{diag}\left\{\frac{1}{\pi_i} K_{x_j, h}(x_i)\right\}_{i \in s},
\]
the local linear regression estimator based on the sample is
\begin{align}
	\widehat{m}_j = e_1^\top \left(\mathcal{X}_s^\top \mathcal{W}_s \mathcal{X}_s\right)^{-1} \mathcal{X}_s^\top \mathcal{W}_s Y_s = w_s Y_s.
	\label{m^(x)}
\end{align}

Let $\mathcal{I}$ be a compact subset of $\mathbb{S}_{x_j,h}$, and let $K_{x_j,h}$ be a continuous associated kernel in the sense of Definition~\ref{definition:associated_kernel}. We denote the centred moments
\begin{equation}\label{Akxh}
	A_k(x_j,h) = \mathbb{E}\left\{(\mathcal{Z}_{x_j,h}-x_j)^k \right\}, \quad k \geq 0,
\end{equation}
so that $A_0(x_j,h) = 1$, $A_1(x_j,h) = A(x_j,h)$, and $A_2(x_j,h) = B(x_j,h) + A(x_j,h)^2$ by the variance decomposition. We impose the following regularity conditions.

\begin{specialassumption}
	For all $i \in s$, $\pi_i = n/N$, where $N$ is the size of the finite population. The general estimator theory in Section~\ref{subsec:asymptotics} holds under the weaker conditions of Assumptions~\ref{B1}--\ref{B7}.
	\label{A1}
\end{specialassumption}

\begin{specialassumption}
	$m^{(2)} \in \mathcal{C}(\mathbb{S}_{x_j,h})$, and both $f(\cdot)$ and $\sigma^2(\cdot)$ are twice continuously differentiable on $\mathbb{S}_{x_j,h}$.
	\label{A2}
\end{specialassumption}

\begin{specialassumption}
	$f(x_j) \geq f_c > 0$ and $\sigma^2(x_j) \leq \sigma_c^2$ for all $x_j \in \mathcal{I}$.
	\label{A3}
\end{specialassumption}

\begin{specialassumption}
	There exists $r_2 = r_2(K_{x_j,h}) > 0$ and $0 \leq c_2(x_j) < \infty$ such that
	\[
	\|K_{x_j,h}\|_2^2 = \int_{\mathbb{S}_{x_j,h} \cap \mathbb{T}}^{} K_{x_j,h}^2(s)\,ds \leq c_2(x_j)\,h^{-r_2}.
	\]
	\label{A4}
\end{specialassumption}

\begin{specialassumption}
	$h \to 0$ and $nh^{r_2} \to \infty$ as $n \to \infty$.
	\label{A5}
\end{specialassumption}

\begin{specialassumption}
	The bias function of the associated kernel satisfies $A(x_j,h)^2 = o\left\{B(x_j,h)\right\}$ as $h \to 0$, so that
	\[
	A_2(x_j,h) = B(x_j,h)\{1+o(1)\}
	\quad\text{and}\quad
	A_1(x_j,h) = o\!\left\{A_2(x_j,h)^{1/2}\right\}.
	\]
	\label{A6}
\end{specialassumption}

\begin{remark}
	Assumptions~\ref{A1} clarifies the role of the inclusion probabilities $\pi_i$ within the proof of Theorem~\ref{thrm} and Assumptions~\ref{B1}--\ref{B7}. Assumption~\ref{A2}--\ref{A3} were used in \citet{chen2002local} in the context of local linear smoothing, while \ref{A4} and~\ref{A5} were used in \citet{kokonendji2018, somes2024}. Assumption~\ref{A6} is an additional assumption that involves associated kernels.
\end{remark}

\begin{theorem}\label{thrm}
	Under Assumptions~\textnormal{\ref{A1}--\ref{A6}}, for $x_j \in \mathcal{I} \subset \mathbb{S}_{x_j,h}$, the estimator~\eqref{m^(x)} has the following asymptotic conditional bias and variance as $n \to \infty$:
	\begin{align}
		\operatorname{Bias}\!\left(\widehat{m}_j \mid x\right)
		= \frac{1}{2}\,m^{(2)}(x_j)\,A_2(x_j,h)
		+ o\!\left\{A_2(x_j,h)\right\}
		+ \mathcal{O}\!\left(\frac{1}{nh^{r_2}}\right),
		\label{Bias_m}
	\end{align}
	where $A_2(x_j,h) = \mathbb{E}\left\{(\mathcal{Z}_{x_j,h}-x_j)^2\right\} = B(x_j,h) + A(x_j,h)^2$, and
	\begin{align}
		\operatorname{Var}\!\left(\widehat{m}_j \mid x\right)
		= \frac{\sigma^2(x_j)}{n\,f(x_j)}\,\|K_{x_j,h}\|_2^2
		+ o\!\left(\frac{1}{nh^{r_2}}\right).
		\label{Var_m}
	\end{align}
\end{theorem}

The following three examples illustrate the theorem for the kernels widely used in this paper, including two classical symmetric cases.

\begin{example}
	The following kernels satisfy Assumptions~\ref{A1}--\ref{A6}.
	\begin{enumerate}
		\item \textit{Gaussian kernel.} For $x_j \in \mathbb{T}=\mathbb{R}$ and $h>0$, the classical associated Gaussian kernel on $\mathbb{S}_{x_j,h}=\mathbb{R}$ is:
		\begin{equation}
			K^{\mathcal{N}}_{x_j, h}(t) = \frac{1}{h\sqrt{2\pi}}\exp\!\left\{-\frac{1}{2}\left(\frac{t - x_j}{h}\right)^2\right\}\mathds{1}_{\mathbb{R}}(t).
			\label{gaus}
		\end{equation}
		Under Assumptions~\ref{A4}--\ref{A6} with $r_2 = 1$, we have $A_2(x_j, h) = h^2$ and $\left\|K^{\mathcal{N}}_{x_j,h}\right\|_2^2 = 1/(h\sqrt{2\pi})$, so that
		\begin{align*}
			\mathrm{Bias}\left(\widehat{m}^{\mathcal{N}}_j \mid x\right) &= \frac{1}{2}m^{(2)}(x_j)\,h^2 + o(h^2) + \mathcal{O}\!\left(\frac{1}{nh}\right), \\
			\mathrm{Var}\left(\widehat{m}^{\mathcal{N}}_j \mid x\right) &= \frac{\sigma^2(x_j)}{nf(x_j)\,h\sqrt{2\pi}} + o\!\left(\frac{1}{nh}\right).
		\end{align*}
		See \citet{fan1992variable, fan1993local, fan1996gijbels} for further results on classical kernels in local estimation and \citet{dolele2017mode, kokonendji2018}.
		
		\item \textit{Epanechnikov kernel.} For $x_j \in \mathbb{T}=\mathbb{R}$ and $h>0$, the Epanechnikov kernel on $\mathbb{S}_{x_j,h}=[x_j-h,x_j+h]$ is given by
		\begin{equation}
			K^{\mathcal{E}}_{x_j,h}(t) = \frac{3}{4h}\left\{1 - \left(\frac{t - x_j}{h}\right)^2\right\} \mathds{1}_{\left|t-x_j\right| \leq h}.
			\label{epan}
		\end{equation}
		Under Assumptions~\ref{A4}--\ref{A6} with $r_2 = 1$, we have $A_2(x_j,h) = h^2/5$ and $\left\|K^{\mathcal{E}}_{x_j,h}\right\|_2^2 = 3/(5h)$, so that
		\begin{align*}
			\mathrm{Bias}\left(\widehat{m}^{\mathcal{E}}_j \mid x\right) &= \frac{1}{2}m^{(2)}(x_j)\frac{h^2}{5} + o(h^2) + \mathcal{O}\!\left(\frac{1}{nh}\right), \\
			\mathrm{Var}\left(\widehat{m}^{\mathcal{E}}_j \mid x\right) &= \frac{3\,\sigma^2(x_j)}{5\,n f(x_j)\,h} + o\!\left(\frac{1}{nh}\right).
		\end{align*}
		Note that when $x_j < h$, the support $[x_j - h, x_j + h]$ extends to negative values, so that this kernel is not appropriate when the auxiliary variable is restricted to $[0,\infty)$; see Section~\ref{sec:ELLL} for a discussion of the practical consequences.
		
		\item \textit{Gamma kernel.} For $x_j \geq 0$ and $h>0$, the gamma kernel on $\mathbb{S}_{x_j,h}=[0,\infty) = \mathbb{T}$ is given by
		\begin{equation}
			K^{\mathcal{G}}_{x_j, h}(t) = \frac{t^{x_j/h}}{\Gamma(1 + x_j/h)\,h^{1 + x_j/h}}\exp\!\left(\frac{-t}{h}\right)\mathds{1}_{[0,\infty)}(t),
			\label{Gam}
		\end{equation}
		where $\Gamma(\cdot)$ denotes the gamma function. This kernel corresponds to the density of a gamma $\mathcal{G}(x_j/h + 1,\, h)$ distribution. Under Assumptions~\ref{A4}--\ref{A6} with $r_2 = 1/2$ \citep{dolele2017mode}, we have $A_2(x_j, h) = h(x_j + 2h)$ and
		\[
		\left\|K^{\mathcal{G}}_{x_j,h}\right\|_2^2 = \frac{\Gamma(2x_j/h + 1)}{h\, 2^{2x_j/h + 1}\,\Gamma^2(x_j/h + 1)},
		\]
		so that for $x_j \geq 0$,
		\begin{align*}
			\mathrm{Bias}\left(\widehat{m}^{\mathcal{G}}_j \mid x\right) &= \frac{1}{2}m^{(2)}(x_j)\left\{h(x_j + 2h)\right\} + o\left\{h(x_j + 2h)\right\} + \mathcal{O}\!\left(\frac{1}{nh^{1/2}}\right), \\
			\mathrm{Var}\left(\widehat{m}^{\mathcal{G}}_j \mid x\right) &= \frac{\sigma^2(x_j)\,\Gamma(2x_j/h + 1)}{nf(x_j)\,h\,2^{2x_j/h+1}\,\Gamma^2(x_j/h+1)} + o\!\left(\frac{1}{nh^{1/2}}\right).
		\end{align*}
		We note that $o\{h(x_j + 2h)\} = o(h^2)$ when $x_j = 0$, and $o(h)$ for $x_j > 0$. These results extend those of \citet{chen2002local} to the survey sampling context: the key difference is the incorporation of design weights $1/\pi_i$ in the weighted least squares criterion \eqref{eq:WLS}, which modifies the finite-sample properties of the estimator while leaving the asymptotic bias and variance structure unchanged.
	\end{enumerate}
\end{example}

To ensure numerical stability when the sample $s$ is a strict subset of $U$, we use the ridge-adjusted estimator
\begin{equation}
	\widehat{m}^1_j = e_1^\top \left(\mathcal{X}_s^\top \mathcal{W}_s \mathcal{X}_s + \mathrm{diag}\!\left\{\frac{\delta}{N^2}\right\}_{i=1}^{2}\right)^{-1} \mathcal{X}_s^\top \mathcal{W}_s Y_s = w^1_s Y_s,
	\label{m^11_j}
\end{equation}
where $\delta > 0$ is a small regularization constant. Substituting \eqref{m^11_j} for \eqref{m_j} in \eqref{y*}, the model-assisted estimator of the finite population mean is
\begin{align}
	\widetilde{y}_{MA} &= \frac{1}{N}\sum_{j \in U}\widehat{m}^1_j + \frac{1}{N}\sum_{j \in s}\frac{y_j - \widehat{m}^1_j}{\pi_j} \label{y_tilde_MA}\\
	&= \frac{1}{N}\sum_{j \in s}\left\{\frac{1}{\pi_j} + \sum_{i \in U}\left(1 - \frac{\mathds{1}_i}{\pi_i} \right)w^1_s e_j\right\} y_j = \frac{1}{N}\sum_{j \in s} w_{js}\, y_j, \nonumber
\end{align}
where $w_{js}$ summarizes the contribution of the auxiliary variable to the estimation weight of unit $j$. The asymptotic properties of $\widetilde{y}_{MA}$ are established in the following subsection.

\subsection{Assumptions}\label{subsec:assumptions}

Assume that $m_j = f(\overline{y}_j, 0)$ and $\widehat{m}^1_j = f(\widehat{\overline{y}}_j, \delta)$, and let
\[
\overline{y}_j = \left[\overline{y}_{jg}\right]_{g=1}^{5} = \left[\frac{1}{N}\sum_{k \in U} K_{x_j,h}(x_k)\, z^+_{jgk}\right]_{g=1}^{5}, \]
\[ \widehat{\overline{y}}_j = \left[\widehat{\overline{y}}_{jg}\right]_{g=1}^{5} = \left[\frac{1}{N}\sum_{k \in U} K_{x_j,h}(x_k)\, z^+_{jgk}\, \frac{\mathds{1}_k}{\pi_k}\right]_{g=1}^{5},
\]
with
\[
z^+_{jgk} = \begin{cases}
	(x_k - x_j)^{g-1} & \text{if } g \leq 3, \\
	y_k(x_k - x_j)^{g-4} & \text{if } g > 3.
\end{cases}
\]
Therefore, $m_j$ and its sample-based counterpart are
\[
m_j = \frac{\overline{y}_{j3}\overline{y}_{j4} - \overline{y}_{j2}\overline{y}_{j5}}{\overline{y}_{j1}\overline{y}_{j3} - \overline{y}^2_{j2}}, \qquad
\widehat{m}_j = \frac{\widehat{\overline{y}}_{j3}\widehat{\overline{y}}_{j4} - \widehat{\overline{y}}_{j2}\widehat{\overline{y}}_{j5}}{\widehat{\overline{y}}_{j1}\widehat{\overline{y}}_{j3} - \widehat{\overline{y}}^2_{j2}}.
\]
For
\[
z_{jk} = \sum_{g=1}^{5} \frac{\partial \widehat{m}^1_j}{\partial(\widehat{\overline{y}}_{jg})}\bigg|_{\widehat{\overline{y}}_j = \overline{y}_j,\, \delta=0} K_{x_j,h}(x_k)\, z^+_{jgk},
\]
we have
\begin{equation}
	R_{jN} = \widehat{m}^1_j - m_j - \frac{1}{N}\sum_{k \in U} z_{jk}\left(\frac{\mathds{1}_k}{\pi_k} - 1\right) - \frac{\partial \widehat{m}^1_j}{\partial \delta}\bigg|_{\widehat{\overline{y}}_j = \overline{y}_j,\, \delta=0} \frac{\delta}{N^2}.
	\label{Rj}
\end{equation}

\begin{assumptions}\label{B1}
	The errors are independent, have zero mean, finite variance, and compact support, uniformly for all $N$.
\end{assumptions}

\begin{assumptions}\label{B2}
	The values $x_j$ are fixed with respect to the superpopulation model \eqref{xi}, and $f$ is the pdf with support $\mathcal{I}$, a compact subset of $\mathbb{S}_{x_j,h} \subseteq \mathbb{R}$, satisfying $f(x_j) > 0$ for all $x_j \in \mathbb{S}_{x_j,h}$.
\end{assumptions}

\begin{assumptions}\label{B3}
	The regression function $m$ is continuous with continuous second derivatives, and the variance function $\sigma^2(\cdot)$ is continuous and strictly positive.
\end{assumptions}

\begin{assumptions}\label{B4}
	The kernel $K_{x_j,h}$ is continuous and not necessarily symmetric, with support $\mathbb{S}_{x_j,h} \subseteq \mathbb{R}$. 
\end{assumptions}

\begin{assumptions}\label{B5}
	As $N \to \infty$, $n/N \to \pi \in (0,1)$ and $h \to 0$.
\end{assumptions}

\begin{assumptions}\label{B6}
	For all $N$, $\min_{j \in U} \pi_j \geq \lambda > 0$, $\min_{i,j \in U} \pi_{ij} \geq \lambda^* > 0$, and
	\[
	\limsup_{N \to \infty}\, n \max_{\substack{i,j \in U \\ i \neq j}} |\pi_{ij} - \pi_i\pi_j| < \infty.
	\]
\end{assumptions}

\begin{assumptions}\label{B7}
	The following moment conditions on the sampling indicators hold:
	\begin{align*}
		&\limsup_{N \to \infty}\, n^2 \max_{(j_1,j_2,j_3,j_4) \in D_{4,N}} \left|\mathbb{E}_p \left[(\mathds{1}_{j_1}-\pi_{j_1})(\mathds{1}_{j_2}-\pi_{j_2})(\mathds{1}_{j_3}-\pi_{j_3})(\mathds{1}_{j_4}-\pi_{j_4})\right]\right| < \infty, \\
		&\lim_{N \to \infty} \max_{(j_1,j_2,j_3,j_4) \in D_{4,N}} \left|\mathbb{E}_p \left[(\mathds{1}_{j_1}\mathds{1}_{j_2}-\pi_{j_1,j_2})(\mathds{1}_{j_3}\mathds{1}_{j_4}-\pi_{j_3,j_4})\right]\right| = 0, \\
		&\limsup_{N \to \infty}\, n \max_{(j_1,j_2,j_3) \in D_{3,N}} \left|\mathbb{E}_p \left[(\mathds{1}_{j_1}-\pi_{j_1})^2(\mathds{1}_{j_2}-\pi_{j_2})(\mathds{1}_{j_3}-\pi_{j_3})\right]\right| < \infty,
	\end{align*}
	where $D_{t,N}$ denotes the set of all distinct $t$-tuples $(j_1, \ldots, j_t)$ from $U$.
\end{assumptions}

\begin{remark}
	Assumptions~\ref{B1}--\ref{B3} and~\ref{B5} and~\ref{B7} are standard in the model-assisted estimation literature; see for example \citet{breidt2000local}, \citet{breidt2005model}, \citet{breidt2017model}, \citet{goga2014efficient}, \citet{dagdoug2023model}, and \citet{mostafa2019finite}. Assumption~\ref{B4} broadens the kernel class by allowing asymmetric associated kernels \citep{esstafa2026properties, stats4010013, wansouwe2016ake, funke2025uniform}.
\end{remark}

We consider the asymptotic framework of \citet{breidt2000local}, where both the population and sample sizes increase to infinity ($N \to \infty$). Two sources of randomness are distinguished. The first is due to the sampling design, with probability measure $p$ and associated notation $\mathcal{O}_p$, $o_p$, and $\mathbb{E}_p(\cdot)$. The second arises from the superpopulation model, in which the finite population is treated as a realization, and is characterized by the measure $x_i$.

\subsection{Asymptotic properties}\label{subsec:asymptotics}

We now establish the asymptotic properties of the estimator $\widetilde{y}_{MA}$ defined in \eqref{y_tilde_MA}, where $\widehat{\mathrm{Var}}(\cdot)$ defined in \eqref{Var^} denotes the empirical variance.
\begin{theorem}\label{thm:unbiased}
	Under Assumptions~\textnormal{\ref{B1}--\ref{B7}}, the estimator \eqref{y_tilde_MA} is asymptotically design-unbiased in the sense that
	\[
	\lim_{N \to \infty} \mathbb{E}_p\left(\widetilde{y}_{MA} - \overline{y}\right) = 0 \quad \xi\text{-probability 1,}
	\]
	and design-consistent in the sense that for all $\epsilon>0,$
	\[
	\lim_{N \to \infty} \mathbb{E}_p\left(\mathds{1}_{\{|\widetilde{y}_{MA} - \overline{y}| > \epsilon\}}\right) = 0 \quad \xi\text{-probability 1}.
	\]
\end{theorem}

\begin{proof}
	Writing
	\begin{align*}
		\widetilde{y}_{MA} - \overline{y} &= \frac{1}{N}\sum_{j \in U}(y_j - m_j)\left(\frac{\mathds{1}_j}{\pi_j} - 1\right) + \frac{1}{N}\sum_{j \in U}(\widehat{m}^1_j - m_j)\left(1 - \frac{\mathds{1}_j}{\pi_j}\right),
	\end{align*}
	so that
	\begin{align}\label{E}
		\mathbb{E}_p \left(|\widetilde{y}_{MA} - \overline{y}|\right) &\leq \mathbb{E}_p \left|\frac{1}{N}\sum_{j \in U}(y_j - m_j)\left({\mathds{1}_j}{\pi_j}^{-1} - 1\right)\right| \nonumber \\
		&\quad + \left[\mathbb{E}_p \left\{\sum_{j \in U}\frac{(\widehat{m}^1_j - m_j)^2}{N}\right\} \mathbb{E}_p \left\{\sum_{j \in U}\frac{\left(1 - {\mathds{1}_j}{\pi_j}^{-1}\right)^2}{N}\right\}\right]^{1/2}.
	\end{align}
	Under Assumptions~\ref{B1}--\ref{B7}, $\limsup_{N \to \infty}\sum_{j \in U}(y_j - m_j)^2/N < \infty$ by \citet[Theorem~1]{breidt2000local}, and the first term on the right-hand side of \eqref{E} converges to 0 by Lemma~\ref{Lemme2} ($\textnormal{iii}$). Under Assumption~\ref{B6},
	\[
	\mathbb{E}_p \left\{\sum_{j \in U}\frac{\left(1 - {\mathds{1}_j}{\pi_j}^{-1}\right)^2}{N}\right\} = \sum_{j \in U}\frac{\pi_j(1-\pi_j)}{N\pi_j^2} \leq \frac{1}{\lambda}.
	\]
	Furthermore, by Lemma~\ref{Lemme4}, $\limsup_{N \to \infty}\sum_{j \in U}\mathbb{E}_p(\widehat{m}^1_j - m_j)^2/N \to 0$, so the entire expression in \eqref{E} tends to zero by Markov's inequality, completing the proof.
\end{proof}

\begin{theorem}\label{thm:MSE}
	Under Assumptions~\textnormal{\ref{B1}--\ref{B7}},
	\begin{equation}
		n\,\mathbb{E}_p \left\{(\widetilde{y}_{MA} - \overline{y})^2\right\} = \frac{n}{N^2}\sum_{i,j \in U}(y_i - m_i)(y_j - m_j)\frac{\pi_{ij} - \pi_i\pi_j}{\pi_i\pi_j} + o(1).
		\label{thm:MSE:eq}
	\end{equation}
\end{theorem}

\begin{proof}
	Following \citet[Theorem~2]{breidt2000local}, let
	\begin{align*}
		n\,\mathbb{E}_p \left\{(\widetilde{y}_{MA} - \overline{y})^2\right\} &= \mathbb{E}_p \left[\left\{\sqrt{n}\sum_{j \in U}\frac{y_j - m_j}{N}\left(\frac{\mathds{1}_j}{\pi_j} - 1\right) + \sqrt{n}\sum_{j \in U}\frac{m_j - \widehat{m}^1_j}{N}\left(\frac{\mathds{1}_j}{\pi_j} - 1\right)\right\}^2\right] \\
		&= \mathbb{E}_p \left\{(L_N + M_N)^2\right\} = \mathbb{E}_p (L_N^2) + \mathbb{E}_p (M_N^2) + 2\mathbb{E}_p (L_N M_N),
	\end{align*}
	where
	\[
	L_N = \sqrt{n}\sum_{j \in U}\frac{y_j - m_j}{N}\left(\frac{\mathds{1}_j}{\pi_j} - 1\right), \qquad M_N = \sqrt{n}\sum_{j \in U}\frac{m_j - \widehat{m}^1_j}{N}\left(\frac{\mathds{1}_j}{\pi_j} - 1\right).
	\]
	Then
	\begin{align*}
		\mathbb{E}_p(L_N^2) &= \frac{n}{N^2}\sum_{i,j \in U}(y_i - m_i)(y_j - m_j)\frac{\pi_{ij} - \pi_i\pi_j}{\pi_i\pi_j} \\
		&\leq \left(\frac{1}{\lambda} + \frac{n\max_{i,j \in U:\, i \neq j}|\pi_{ij} - \pi_i\pi_j|}{\lambda^2}\right)\sum_{j \in U}\frac{(y_j - m_j)^2}{N},
	\end{align*}
	so $\limsup_{N \to \infty}\mathbb{E}_p (L_N^2) < \infty$ by Assumption~\ref{B6}. By Lemma~\ref{Lemme5}, $\mathbb{E}_p (M_N^2) = o(1)$, and by the Cauchy--Schwarz inequality,
	\[
	\mathbb{E}_p (L_N M_N) \leq \left\{\mathbb{E}_p (L_N^2)\mathbb{E}_p (M_N^2)\right\}^{1/2} = o(1).
	\]
	Therefore $n\,\mathbb{E}_p (\widetilde{y}_{MA} - \overline{y})^2 = \mathbb{E}_p (L_N^2) + o(1)$, which establishes the result.
\end{proof}

The asymptotic mean squared error in ~\eqref{thm:MSE:eq} can be estimated consistently under mild conditions.

\begin{theorem}\label{thm:variance}
	Under Assumptions~\textnormal{\ref{B1}--\ref{B7}},
	\[
	\lim_{N \to \infty} n\,\mathbb{E}_p\left|\widehat{\mathrm{Var}}(\widetilde{y}_{MA}) - AMSE(\widetilde{y}_{MA})\right| = 0,
	\]
	where
	\begin{equation}\label{Var^}
		\widehat{\mathrm{Var}}(\widetilde{y}_{MA}) = \frac{1}{N^2}\sum_{i,j \in U}\frac{(y_i - \widehat{m}^1_i)}{\pi_i} \frac{(y_j - \widehat{m}^1_j)}{\pi_j}\frac{\pi_{ij} - \pi_i\pi_j}{\pi_{ij}}{\mathds{1}_i \mathds{1}_j},
	\end{equation}
	and
	\[
	AMSE(\widetilde{y}_{MA}) = \frac{1}{N^2}\sum_{i,j \in U}(y_i - m_i)(y_j - m_j)\frac{\pi_{ij} - \pi_i\pi_j}{\pi_i\pi_j}.
	\]
	Consequently, $\widehat{\mathrm{Var}}(\widetilde{y}_{MA})$ is asymptotically unbiased and consistent for $AMSE(\widetilde{y}_{MA})$ with respect to the sampling design.
\end{theorem}

\begin{proof}
	The result follows from \citet[Theorem~3]{breidt2000local} upon replacing the classical kernel regression estimator by the associated kernel estimator $\widehat{m}^1_j$.
\end{proof}

The asymptotic normality of \eqref{y_tilde_MA} follows from that of the generalized difference estimator $\overline{y}^*$.

\begin{theorem}\label{thm:normality}
	Under Assumptions~\textnormal{\ref{B1}--\ref{B7}}, let $\overline{y}^{\ast}$ and $\mathrm{Var}_p(\overline{y}^{\ast})$ be as defined in \eqref{y*} and \eqref{var_y_bar*}. As $N \to \infty$,
	\[
	\frac{\overline{y}^* - \overline{y}}{\mathrm{Var}_p^{1/2}(\overline{y}^*)} \xrightarrow{\mathcal{L}} \mathit{N}(0,1) \implies \frac{\widetilde{y}_{MA} - \overline{y}}{\widehat{\mathrm{Var}}^{1/2}(\widetilde{y}_{MA})} \xrightarrow{\mathcal{L}} \mathit{N}(0,1),
	\]
	where $\widehat{\mathrm{Var}}(\widetilde{y}_{MA})$ is as defined in~\eqref{Var^}.
\end{theorem}

\begin{proof}
	From the proof of Theorem~\ref{thm:MSE},
	\[
	\widetilde{y}_{MA} - \overline{y} = \sum_{j \in U}\frac{y_j - m_j}{N}\left(\frac{\mathds{1}_j}{\pi_j}- 1\right) + o_p\!\left(\frac{1}{\sqrt{n}}\right) = \overline{y}^* - \overline{y} + o_p\!\left(\frac{1}{\sqrt{n}}\right).
	\]
	By Theorem~\ref{thm:variance}, $\widehat{\mathrm{Var}}(\widetilde{y}_{MA})/AMSE(\widetilde{y}_{MA}) \xrightarrow{p} 1$, and the result follows.
\end{proof}

\begin{corollary}\label{cor:SRSWOR}
	Under \textnormal{SRSWOR} and Assumptions~\textnormal{\ref{B1}--\ref{B7}}, as $N \to \infty$,
	\[
	\frac{\widetilde{y}_{MA} - \overline{y}}{\widehat{\mathrm{Var}}^{1/2}(\widetilde{y}_{MA})} \xrightarrow{\mathcal{L}} \mathit{N}(0,1),
	\]
	where
	\[
	\widehat{\mathrm{Var}}(\widetilde{y}_{MA}) = \left(1 - \frac{n}{N}\right)\frac{\sum_{j \in s}(y_j - \widehat{m}^1_j)^2 - n^{-1}\left\{\sum_{j \in s}(y_j - \widehat{m}^1_j)\right\}^2}{n(n-1)}.
	\]
\end{corollary}

\begin{proof}
	Under Assumptions~\ref{B1}--\ref{B7} and by Lemma~\ref{Lemme2}$\left(\textnormal{iv}\right)$, $\limsup_{N \to \infty}\sum_{j \in U}(y_j - m_j)^4/N < \infty$. The variance estimator formula follows from \citet[Corollary~1]{breidt2000local}, and asymptotic normality of \eqref{y*} follows from \citet[Theorem~3.2]{Thompson1997}. The result then follows from Theorem~\ref{thm:normality}.
\end{proof}

We further establish that $\widetilde{y}_{MA}$ is robust in the sense that it asymptotically attains the Godambe--Joshi (GJ) lower bound.

\begin{theorem}\label{thm:robustess}
	Under Assumptions~\textnormal{\ref{B1}--\ref{B7}}, $\widetilde{y}_{MA}$ asymptotically attains the GJ lower bound, in the sense that
	\[
	n \mathbb{E}\left(\widetilde{y}_{MA}-\overline{y}\right)^{2} = \frac{n}{N^{2}}\sum_{j \in U} \sigma^2(x_j) \frac{1-\pi_j}{\pi_j} + o(1).
	\]
\end{theorem}

\begin{proof}
	The result follows from \citet[Theorem~5]{breidt2000local}, taking into account Lemma~\ref{Lemme6}, Lemma~\ref{Lemme5} for all $i \in U$ and $s$ and across realizations from $\xi$, Assumption~\ref{B3}, and the Cauchy--Schwarz inequality.
\end{proof}

Bandwidth selection is carried out by minimizing the survey cross-validation criterion of \citet{opsomer2005selecting}:
\begin{equation}
	LSCV(h) = \frac{1}{n}\sum_{i,j \in s}\frac{\pi_{ij}-\pi_i \pi_j}{\pi_{ij}}\frac{y_i - \widehat{m}^1_{-i}}{\pi_i}\frac{y_j - \widehat{m}^1_{-j}}{\pi_j},
	\label{CV(h)_2}
\end{equation}
where $\widehat{m}^1_{-j}$ denotes the estimator $\widehat{m}^1_j$ \eqref{m^11_j} computed with observation $(x_j, y_j)$ left out, for $j \in s$. The finite-sample performance of $\widetilde{y}_{MA}$ under the three kernels of Example~1 is assessed in Section~\ref{sec:ELLL}.

\section{Simulations and illustration}\label{sec:ELLL}
  
In this section, we provide numerical results to evaluate the performance 
of the proposed model-assisted estimator of the mean $\tilde{y}_{MA}$ 
\eqref{y_tilde_MA} using local linear regression \eqref{m^11_j} with asymmetric gamma kernels \eqref{Gam}. Comparisons are also made 
with the symmetric Gaussian \eqref{gaus} and Epanechnikov kernels \eqref{epan} and also the HT estimator \eqref{bar_y_HT}. All computations were performed using \textsf{R} software 
\citep{R20} on a computer equipped with an Intel Core i5 processor 
(2.30~GHz) with 4~GB of RAM. The optimal bandwidth is selected by 
minimizing the cross-validation criterion \eqref{CV(h)_2}; see 
\cite{opsomer2005selecting}.

The performance of the mean estimators is assessed through the following 
criteria:
\begin{equation}\label{MSE}
\widehat{MSE} = \frac{1}{C}\sum_{c=1}^{C}\left(\widetilde{y}_{MA}(s_c) - \overline{y}\right)^{2},
\end{equation}
\begin{equation}\label{RB and ER}
RB\!\left(\widetilde{y}_{MA}\right) = \frac{1}{C}\sum_{c=1}^{C}
\left|\frac{\widetilde{y}_{MA}(s_c) - \overline{y}}{\overline{y}}\right|,
\qquad
RE\!\left(\widetilde{y}_{MA}\right) = 
\frac{\displaystyle\sum_{c=1}^{C}
	\left(\widetilde{y}_{MA}(s_c) - \overline{y}\right)^{2}}
{\displaystyle\sum_{c=1}^{C}
	\left(\widehat{\overline{y}}_{\mathrm{HT}}(s_c) - \overline{y}\right)^{2}},
\end{equation}
where $C = 1000$ unless otherwise stated is the number of Monte Carlo replications. The relative error $\textit{RE}$ measures the gain over the HT 
estimator: values below~1 (or below~100\%) indicate an improvement.

\subsection{Simulation design}\label{subsec:Simulations}

We consider eight regression functions (scenario), among those introduced in \citet{breidt2000local} and \citet{somes2024}, defined on their respective domains as follows:
\begin{align*}
	& \text{Linear:}\ m_{1}(x) = \left\{1 + 2(x-0.5)\right\}\mathds{1}_{[0,1]}(x), \\
	& \text{Quadratic:}\ m_{2}(x) = \left\{1 + 2(x-0.5)^{2} \right\}\mathds{1}_{[0,1]}(x), \\
	& \text{Jump:}\ m_{3}(x) = \left[ \left\{1 + 2(x-0.5)\right\}\mathds{1}_{\{x < 0.65\}} + 0.65\,\mathds{1}_{\{x \geq 0.65\}} \right]\mathds{1}_{[0,1]}(x), \\
	& \text{CDF:}\ m_{4}(x) = \Phi\!\left(\frac{1.5 - 2x}{\sigma_0}\right)\mathds{1}_{[0,1]}(x),\quad \text{with } \sigma_0 = 1, \\
	& \text{Exponential:}\ m_{5}(x) = \exp(-8x)\mathds{1}_{[0,\infty)}(x), \\
	& \text{Mixture density 1:}\ m_{6}(x) = \left( 
	\frac{3}{10}\frac{x e^{-3x}}{3^{-2}\Gamma(2)} 
	+ \frac{2}{5}\frac{x^7 e^{-3x}}{3^{-8}\Gamma(8)} 
	+ \frac{3}{10}\frac{x^{10} e^{-3x}}{3^{-11}\Gamma(11)} 
	\right)\mathds{1}_{[0,\infty)}(x), \\
	& \text{Weibull density:}\ m_{7}(x)= \frac{3}{2}\left(\frac{x}{2}\right)^2 e^{-(x/2)^3}\mathds{1}_{[0,\infty)}(x), \\
	& \text{Mixture density 2:}\ m_{8}(x) = \left( 
	\frac{1}{3}\frac{x e^{-2.5x}}{2.5^{-2}} 
	+ \frac{1}{3}\frac{2.5^8 x^7 e^{-2.5x}}{7!} 
	+ \frac{1}{3}\frac{2.5^{25} x^{24} e^{-2.5x}}{24!} 
	\right)\mathds{1}_{[0,\infty)}(x).
\end{align*}

These functions are chosen to represent key challenges in nonparametric regression, including smoothness, discontinuities, boundary effects, and asymmetric supports. 
They cover a broad range of shapes: linear and quadratic ($m_1$, $m_2$), piecewise linear with a jump discontinuity at $x=0.65$ ($m_3$), a cumulative distribution function ($m_4$), exponential decay ($m_5$), and mixture-type densities on $[0,\infty)$ ($m_6$, $m_7$, $m_8$).
The auxiliary variable $x$ is generated from i.i.d.\ $\mathrm{Uniform}(0,1)$ random variables for $m_1$--$m_5$ and $\mathrm{Uniform}(0,6.5)$ for $m_6$--$m_8$; this truncation ensures computational efficiency while preserving the essential mass of the distributions.

The model errors $\varepsilon$ are i.i.d.\ $\mathit{N}(0,\sigma^2)$, independent of $x$, with noise levels $\sigma \in \{0.1,\,0.3,\,0.5\}$. 
A finite population of size $N=1000$ is generated once, and samples of sizes $n \in \{50,100,300\}$ are drawn by SRSWOR, with first-order inclusion probabilities $\pi_i = n/N$.

For each combination $( \sigma, n )$, $C=1000$ independent samples are drawn from the same population. 
The three estimators are computed, and $\widehat{\mathrm{MSE}}$, $\textit{RB}$, and $\textit{RE}$ are evaluated empirically.

Tables~\ref{Ratios des MSE_LLR}--\ref{Ratios des ER_Y_MA} summarize the empirical performances of the three model-assisted estimators under different regression functions, sample sizes and noise levels. Overall, no kernel uniformly dominates across all scenarios. This is expected since the simulation study includes both smooth symmetric functions and asymmetric functions defined on positive supports.

A preliminary remark on theoretical appropriateness is in order. 
All eight regression functions are defined on supports that are either compact $[0,1]$ or semi-infinite $[0,\infty)$, so that the auxiliary variable $x$ is strictly non-negative in every scenario. 
Symmetric kernels such as Gaussian \eqref{gaus} and Epanechnikov \eqref{epan}, when centered near the left boundary $x=0$, assign positive mass outside the support, thereby introducing well-known boundary bias. 
The  gamma kernel, by construction, is the only among the three that is structurally compatible with the support of the data. Numerical results should be interpreted in this light.

Table~\ref{Ratios des MSE_LLR} reports the empirical mean squared errors. For the smooth functions $m_1$ and $m_4$, the three kernels provide very similar performances, although the Gaussian and Epanechnikov kernels often give slightly smaller $\widehat{\mathrm{MSE}}$ values. The case of $m_2$ is particular because the regression function is symmetric around an interior point of the support. In this situation, symmetric kernels are naturally favoured, which explains the better performances of the Gaussian and Epanechnikov kernels.

For the asymmetric regression functions $m_5$--$m_8$, the gamma kernel generally remains competitive and becomes more advantageous in several configurations, especially for moderate and large noise levels. Although it does not systematically produce the smallest $\widehat{\mathrm{MSE}}$, its performances remain relatively stable across most scenarios. By contrast, the Epanechnikov kernel occasionally produces larger values and higher Monte Carlo variability for some mixture-density functions, particularly for $m_6$, $m_7$ and $m_8$.

Table~\ref{Ratios des BR_Y_MA} reports the relative biases. The three kernels yield biases of comparable magnitude in many configurations, especially as the sample size increases. As expected, the relative biases decrease with increasing sample size. For the asymmetric functions $m_5$--$m_8$, the gamma kernel often gives slightly smaller or more stable relative biases, particularly at $n=100$ and $n=300$.

The clearest differences appear in Table~\ref{Ratios des ER_Y_MA}, which reports the relative efficiencies with respect to the Horvitz--Thompson estimator. Several asymmetric scenarios show that the gamma kernel remains below the $100\%$ threshold while the symmetric kernels may exceed it. For example, for $m_6$ with $n=100$ and $\sigma=0.5$, the gamma kernel gives $\textit{RE}=97.54\%$, whereas the Gaussian and Epanechnikov kernels yield $101.00\%$ and $102.50\%$, respectively. Similar patterns are observed for $m_8$ under large noise levels. These results suggest that the gamma kernel provides more stable performances when the auxiliary variable is non-negative and the regression function is strongly asymmetric.

Overall, the simulation results show that the advantage of the gamma kernel is not universal, but becomes more visible for asymmetric and boundary-sensitive regression functions. In such situations, its support-adapted structure helps reduce some of the instability observed with symmetric kernels.


\begin{landscape} 
\begin{table}[!htbp]
	\centering
	\caption{Estimated mean squared errors, reported as $10^3\times \widehat{MSE}$, with standard deviations in parentheses over $C=1000$ replications. The smallest value in each row (for each $\sigma$) is shown in bold.}
	\label{Ratios des MSE_LLR}
	\small
	\setlength{\tabcolsep}{4pt}
	\renewcommand{\arraystretch}{1.05}	
	\begin{tabular}{c l rrr rrr rrr}
		\toprule
		& & \multicolumn{3}{c}{$\sigma=0.1$}
		& \multicolumn{3}{c}{$\sigma=0.3$}
		& \multicolumn{3}{c}{$\sigma=0.5$} \\
		\cmidrule(lr){3-5} \cmidrule(lr){6-8} \cmidrule(lr){9-11}
		$n$ & Function 
 	&	$\widehat{\mathrm{MSE}}_{\mathcal{G}}$ & $\widehat{\mathrm{MSE}}_{\mathcal{N}}$ & $\widehat{\mathrm{MSE}}_{\mathcal{E}}$
 	&	$\widehat{\mathrm{MSE}}_{\mathcal{G}}$ & $\widehat{\mathrm{MSE}}_{\mathcal{N}}$ & $\widehat{\mathrm{MSE}}_{\mathcal{E}}$
 	&	$\widehat{\mathrm{MSE}}_{\mathcal{G}}$ & $\widehat{\mathrm{MSE}}_{\mathcal{N}}$ & $\widehat{\mathrm{MSE}}_{\mathcal{E}}$\\
 		\midrule
 		\multirow{8}{*}{50}
 		& $m_1$  & 0.20 (0.28) &  \textbf{0.19 (0.26)}& 0.19 (0.27)  & \textbf{1.70 (2.39)}& 1.72 (2.42) & 1.74 (2.44) & \textbf{4.76 (6.70)} & 4.79 (6.74) &  4.83 (6.79)\\
 		& $m_2$  & 0.67 (0.96) & 0.20 (0.30) & \textbf{0.20 (0.29) }&  2.15 (3.16) & 1.81 (2.60) & \textbf{1.80 (2.56) } & 5.19 (7.53) & \textbf{4.95 (7.06)} & 4.96 (7.04) \\
 		& $m_3$  &  0.49 (0.69) &\textbf{0.27 (0.38)}& 0.85 (1.08)  & 1.99 (2.78)  & \textbf{1.87 (2.70)} & 1.88 (2.73) & 5.01 (6.99)& 4.98 (7.14)& \textbf{4.96 (7.14)} \\
 		& $m_4$  & 0.22 (0.34)& \textbf{0.19 (0.27)}& 0.19 (0.27)&1.75 (2.49)&\textbf{1.73 (2.42)}& 1.74 (2.45)&\textbf{4.78 (6.77)}&4.80 (6.73)& 4.83 (6.81) \\
 		& $m_5$  &0.59 (0.84) & \textbf{0.22 (0.35)}& 0.23 (0.37)&2.08 (3.07)& \textbf{1.87 (2.81)}& 1.88 (2.82)&5.08 (7.41)&\textbf{5.02 (7.47)}& 5.10 (7.54) \\
 		& $m_{6}$ & 0.25 (0.35) &\textbf{0.24 (0.34)} & 0.26 (0.55) & \textbf{1.80 (2.53)}& 1.89 (3.41)& 2.19 (10.48) & \textbf{4.85 (6.82)} &4.97 (7.39) &6.00 (28.88)  \\
 		& $m_{7}$ &0.69 (0.91)& \textbf{0.25 (0.36)}&0.49 (7.22)& 2.13 (2.90)&\textbf{2.01 (3.63)}&2.58 (12.91)&\textbf{5.18 (7.12)}& 5.43 (9.85)& 6.16 (28.74)\\
 		& $m_{8}$ & \textbf{0.21 (0.30)}&0.23 (0.40)&0.30 (1.46)&\textbf{1.73 (2.45)}& 1.80 (2.64) & 2.17 (10.41)&\textbf{4.76 (6.78)}& 4.88 (6.84) & 5.99 (28.78) \\

 		\midrule
 		
 		\multirow{8}{*}{100}
 		& $m_1$  & 0.09 (0.15)  & \textbf{0.08 (0.12)}  & 0.08 (0.12) & 0.78 (1.14) & \textbf{0.77 (1.11)}& 0.78 (1.12) & \textbf{2.14 (3.13)} & 2.16 (3.10) & 2.19 (3.12) \\
 		& $m_2$  & 0.30 (0.44) & 0.09 (0.13)& \textbf{0.09 (0.13)} & 1.00 (1.46) & 0.79 (1.14) & \textbf{0.79 (1.13)} & 2.39 (3.47) & 2.21 (3.17)& \textbf{2.20 (3.12)}  \\
 		& $m_3$  &  0.22 (0.31)& \textbf{0.10 (0.15)}  & 0.13 (0.49) & 0.91 (1.32) & \textbf{0.83 (1.20)} & 0.85 (3.31)& 2.29 (3.31) &\textbf{2.23 (3.22)} & 2.28 (3.35)\\
 		& $m_4$  & 0.09 (0.14)& \textbf{0.08 (0.12)} & 0.08	(0.12)&0.78 (1.14)&\textbf{0.77 (1.12)}&0.78 (1.13)& \textbf{2.15 (3.14)} &2.16 (3.11) & 2.17 (3.13) \\
 		& $m_5$  &0.26 (0.38) &\textbf{0.09 (0.13)}&0.10 (0.42)&0.96 (1.40)&\textbf{0.79 (1.15)}&0.81 (1.18)&2.57 (3.42)& 2.22 (3.19)& \textbf{2.22 (3.16)}\\
 		& $m_{6}$ & 0.10 (0.15) &  \textbf{0.09 (0.14)}& 0.10 (0.14)& \textbf{0.78 (1.14)} & 0.81 (1.20) & 0.81 (1.19) & \textbf{2.14 (3.11) }&2.21 (3.24) & 2.25 (3.29)  \\
 		& $m_{7}$ &0.26 (0.38) &\textbf{0.09 (0.14)}& 0.69 (0.14)&0.95 (1.37)&\textbf{0.84 (1.24)}&0.85 (1.25)&2.30 (3.36) &2.29 (3.38)& \textbf{2.28 (3.31)} \\
 		& $m_{8}$ &\textbf{0.09 (0.14)} & 0.09 (0.14)&0.09 (0.14)&\textbf{0.78 (1.14)}&0.81 (1.18)&0.81 (1.18)&\textbf{2.15 (3.10)}&2.21 (3.23)& 2.24 (3.25) \\
 		
 		\midrule
 		
 		\multirow{8}{*}{300}
 		& $m_1$  &  0.11 (0.15) &  \textbf{0.10 (0.14)}& 0.10 (0.14)  & \textbf{0.21 (0.29)}  & 0.21 (0.29)  &  0.22 (0.29) & \textbf{0.63 (0.89)} & 0.63 (0.91) &  0.63 (0.91)\\
 		& $m_2$  & 0.08 (0.12) & \textbf{0.02 (0.03)} & 0.02 (0.03)& 0.26 (0.37)& \textbf{0.20 (0.28)}& 0.20 (0.28)& 0.65 (0.86) & \textbf{0.61 (0.81)} & 0.61 (0.81)  \\
 		& $m_3$  & 0.05 (0.08)&0.02 (0.03)& \textbf{0.02 (0.03)} & 0.24 (0.33)&  \textbf{0.21 (0.31)}& 0.22 (0.32) &0.61 (0.85)& \textbf{0.61 (0.86)} & 0.61 (0.86) \\
 		& $m_4$  & 0.02 (0.04)&0.02 (0.03)& \textbf{0.02 (0.03)}&0.21 (0.30)&\textbf{0.21 (0.30)}&0.21 (0.30)&\textbf{0.59 (0.83)}&0.60 (0.84)& 0.60 (0.84) \\
 		& $m_5$  &0.07 (0.10) & 0.02 (0.03)& \textbf{0.02 (0.03)}&0.25 (0.34)&\textbf{0.22 (0.30)}&0.22 (0.30)&0.64 (0.90)& \textbf{0.60 (0.85)}& 0.60 (0.85) \\
 		& $m_{6}$  & 0.03 (0.04) &0.02 (0.03) &\textbf{0.02 (0.03)} & 0.22 (0.30) & 0.22 (0.31)  & \textbf{0.22 (0.31)} &\textbf{0.61 (0.83)}  & 0.60 (0.85) & 0.61 (0.86) \\
 		& $m_{7}$ &0.07 (0.10) &0.02 (0.03)&\textbf{0.02 (0.03)}& 0.26 (0.38)&\textbf{0.22 (0.31)}&0.22 (0.31)&0.33 (0.49)& \textbf{0.29 (0.42)}& 0.29 (0.42) \\
 		& $m_{8}$ &0.02 (0.03) & 0.02 (0.03)&\textbf{0.02 (0.03)}&\textbf{0.23 (0.32)}&0.23 (0.33)&0.23 (0.33)&\textbf{0.56 (0.79)}& 0.57 (0.79)& 0.58 (0.80) \\
 		
 		\botrule
	\end{tabular}
\end{table}
\end{landscape}

\begin{landscape}
\begin{table}[!htbp]
	\centering
	\caption{Relative biases, reported as $10^2 \times \textit{RB}$, with standard deviations in parentheses.
 	 For each $\sigma$, the smallest value in each row is shown in bold.}
	\label{Ratios des BR_Y_MA}
	\small
	\setlength{\tabcolsep}{2.5pt}
	\renewcommand{\arraystretch}{1.20}
	\begin{tabular}{c l rrr rrr rrr}
		\toprule
		& & \multicolumn{3}{c}{$\sigma=0.1$}
		& \multicolumn{3}{c}{$\sigma=0.3$}
		& \multicolumn{3}{c}{$\sigma=0.5$} \\
		\cmidrule(lr){3-5} \cmidrule(lr){6-8} \cmidrule(lr){9-11}
			$n$ & Function 
 	&	$\textit{RB}_{\mathcal{G}}$ & $\textit{RB}_{\mathcal{N}}$ & $\textit{RB}_{\mathcal{E}}$
 	&	$\textit{RB}_{\mathcal{G}}$ & $\textit{RB}_{\mathcal{N}}$ & $\textit{RB}_{\mathcal{E}}$
 	&	$\textit{RB}_{\mathcal{G}}$ & $\textit{RB}_{\mathcal{N}}$ & $\textit{RB}_{\mathcal{E}}$\\
 		\midrule
 		
 		\multirow{8}{*}{50}
 		& $m_1$  & 1.12 (0.86) & \textbf{1.09 (0.84)} & 1.10 (0.84) & \textbf{3.28 (2.51)} & 3.29 (2.54) & 3.31 (2.54) & \textbf{5.47 (4.22)} & 5.49 (4.23) & 5.52 (4.24) \\
 		& $m_2$  & 1.79 (1.32) & 0.97 (0.75) & \textbf{0.96 (0.75)} & 3.15 (2.43) & 2.88 (2.23) & \textbf{2.87 (2.23)} & 4.87 (3.79) & \textbf{4.76 (3.70)} & 4.77 (3.70) \\
 		& $m_3$  & 1.78 (1.33) &\textbf{1.31 (1.00)} & 1.52 (2.48) &  3.55 (2.71)& \textbf{3.41 (2.66)} & 3.42 (2.68) & 5.61 (4.33) &  4.59 (4.32)   &\textbf{5.57 (4.33)} \\
 		& $m_4$  &1.79 (1.38)  &1.68 (1.27)& \textbf{1.68 (1.27)}&4.98 (3.85) & \textbf{4.97 (3.78)}& 4.97 (3.80)& \textbf{8.21 (6.37)}& 8.29 (6.33)& 8.29 (6.35) \\
 		& $m_5$  & 15.56 (11.60)& \textbf{9.42 (7.53)}& 9.58 (7.69)&29.03 (22.48)& \textbf{27.28 (21.85)}& 27.37 (21.95)&45.47 (35.89)& \textbf{45.26 (35.51)}& 45.50 (35.12) \\
 		& $m_{6}$ &8.32 (6.24) & \textbf{7.98 (6.18)}&8.17 (6.83)& \textbf{22.21 (16.92)}&22.37 (17.62)& 22.90 (21.37) &\textbf{36.54 (28.40)}& 36.83 (28.79)&38.03 (37.06) \\
 		& $m_{7}$ & 13.93 (10.06) &\textbf{8.25 (6.16)}&8.56 (11.27)& 24.44 (18.20) & \textbf{22.86 (18.59)}& 23.79 (23.87)& 38.07 (29.22)& \textbf{37.82 (30.82)}& 38.48 (37.43)\\
 		& $m_{8}$ &\textbf{11.29 (8.51)} &11.66 (8.98)&11.73 (3.83) & \textbf{32.06 (25.11)}&32.90 (25.37)& 33.93 (32.59) &\textbf{53.98 (43.67)}&54.98 (43.74) &57.22 (58.91) \\
 		
 		\midrule
 		
 		\multirow{8}{*}{100}
 		& $m_1$  &0.74 (0.59) &  \textbf{0.73 (0.56)} & 0.74 (0.56) & \textbf{2.21 (1.71)}& 2.21 (1.69) & 2.22 (1.70) & \textbf{3.67 (2.82)} & 3.69 (2.83) & 3.70 (2.84) \\
 		& $m_2$  & 1.17 (0.91) & 0.64 (0.49) & \textbf{0.64 (0.49)} &2.14 (1.66) & 1.92 (1.46) & \textbf{1.91 (1.46)} & 3.32 (2.55) & 3.20 (2.44) & \textbf{3.20 (2.43)}  \\
 		& $m_3$  & 1.18 (0.91) &  \textbf{0.81 (0.64)}&0.87 (0.79)&2.42 (1.82)&\textbf{2.29 (1.76)}&2.32 (1.78)& 3.82 (2.89) & \textbf{3.76 (2.87)} & 3.79 (2.91)\\
 		& $m_4$  &1.17 (0.91) & \textbf{1.11 (0.84)}& 1.11 (0.85)&\textbf{3.31 (2.57)}&3.32 (2.54)&3.33 (2.55)&\textbf{5.52 (4.26)}& 5.54 (4.24)& 5.56 (4.27) \\
 		& $m_5$  & 10.38 (8.01)&\textbf{6.02 (4.69)}& 6.19 (5.35)&19.86 (15.35)&\textbf{18.10 (13.94)}&18.20 (14.13)&31.39 (24.38)& \textbf{30.53 (23.71)}&30.56 (23.80) \\
 		& $m_{6}$ &5.40 (4.17) & \textbf{5.19 (3.92)}&5.22 (3.92)&\textbf{14.62 (11.22)}&14.81 (11.50)&14.87 (11.45)& \textbf{24.25 (18.73)}&24.62 (19.16)& 24.73 (19.30) \\
 		& $m_{7}$ & 8.43 (6.40) & \textbf{5.09 (3.89)}&5.61 (15.28)& 15.96 (12.23)& \textbf{14.98 (11.58)}& 15.08 (11.63)& 24.95 (19.42) & \textbf{24.85 (19.43)}& 24.90 (19.24) \\
 		& $m_{8}$ &\textbf{7.58 (5.81)} &7.62 (5.78) &7.68 (5.79)&\textbf{21.66 (16.74)} &22.08 (17.12)&22.16 (16.98)&\textbf{36.53 (28.61)}& 36.99 (29.20)&37.13 (29.13) \\
 		
 		\midrule
 		
 		\multirow{8}{*}{300}
 		& $m_1$  & \textbf{3.88 (2.92)} & 3.88 (2.94) & 3.89 (2.94)& \textbf{1.16 (0.87)} & 1.16 (0.87) & 1.16 (0.87) &\textbf{2.01 (1.52)} & 2.02 (1.52) & 2.02 (1.52)  \\
 		& $m_2$  & 0.61 (0.47) & 0.33 (0.25) & \textbf{0.31 (0.25)} & 1.11 (0.83) & \textbf{0.98 (0.74)} & 0.98 (0.75) &1.76 (1.28) & \textbf{1.70 (1.25)} & 1.70 (1.25) \\
 		& $m_3$  &  0.57 (0.45) & 0.42 (0.32)& \textbf{0.41 (0.32)}& 1.22 (0.93)& \textbf{1.18 (0.90)}& 1.19 (0.92)& 1.96 (1.50) & \textbf{1.96 (1.50)} & 1.96 (1.50) \\
 		& $m_4$  &0.63 (0.48)&0.60 (0.46)&\textbf{0.60 (0.46)}&1.76 (1.33)& \textbf{1.76 (1.33)}& 1.77 (1.33)&\textbf{2.93 (2.20)}& 2.93 (2.21)& 2.94 (2.21) \\
 		& $m_5$  & 5.48 (4.27)&3.21 (2.41)& \textbf{3.20 (2.84)}&10.55 (7.72)& \textbf{9.66 (7.29)}& 9.71 (7.29)&16.61 (12.82)&\textbf{16.12 (12.56)}& 16.13 (12.57) \\
 		& $m_{6}$  &2.92 (2.16) & 2.65 (1.99)&\textbf{2.64 (1.97)}&7.86 (5.96)&\textbf{7.83 (5.97)} & 7.83 (5.95)&\textbf{12.97 (9.95)}&13.06 (9.99)& 13.10 (10.06)\\
 		& $m_{7}$ &4.36 (3.29) & 2.52 (1.94)&\textbf{2.51 (1.93)}&8.49 (6.49) & \textbf{7.87 (5.96)}& 7.89 (5.97)& 8.57 (6.59)& \textbf{7.59 (5.80)}& 7.61 (5.81) \\
 		& $m_{8}$ &3.89 (3.00) &\textbf{3.80 (2.90)}&3.82 (2.9) &\textbf{11.88 (9.19)} &11.95 (9.27)&12.01 (9.25)&\textbf{18.94 (15.21)} & 19.02 (15.26)& 19.18 (15.41)\\
 		
 		\botrule
	\end{tabular}
\end{table}
\end{landscape}

 \begin{table}[!htbp]
 	\centering
 	\caption{Relative errors (\%) for selected functions across sample sizes and noise levels. 
 	The smallest value in each row is shown in bold.}
 \label{Ratios des ER_Y_MA}
 	\setlength{\tabcolsep}{4pt}
 		\renewcommand{\arraystretch}{1.20}
 	\begin{tabular}{c l rrr rrr rrr}
 		\toprule
 		& & \multicolumn{3}{c}{$\sigma=0.1$} 
 		& \multicolumn{3}{c}{$\sigma=0.3$} 
 		& \multicolumn{3}{c}{$\sigma=0.5$} \\
 		\cmidrule(lr){3-5} \cmidrule(lr){6-8} \cmidrule(lr){9-11}
 	$n$ & Function 
 	&	$\textit{RE}_{\mathcal{G}}$ & $\textit{RE}_{\mathcal{N}}$ & $\textit{RE}_{\mathcal{E}}$
 	&$\textit{RE}_{\mathcal{G}}$ & $\textit{RE}_{\mathcal{N}}$ & $\textit{RE}_{\mathcal{E}}$
 	&$RE_{\mathcal{G}}$ & $\textit{RE}_{\mathcal{N}}$ & $\textit{RE}_{\mathcal{E}}$\\
 		\midrule
 		
 		\multirow{8}{*}{50}
 		& $m_1$  &3.22 & \textbf{3.08} & 3.11 & \textbf{22.42} & 22.70 & 22.87 & \textbf{45.47} & 45.80 & 46.14 \\
 		& $m_2$  & 111.92 & \textbf{34.12} & 34.74 & 104.80 & 87.93 & \textbf{87.59} & 104.13 & \textbf{99.35} & 99.54 \\
 		& $m_3$  & 8.03 &\textbf{4.40} & 13.89 & 27.32 &\textbf{25.66} &25.80  & 49.35 & 49.04   & \textbf{48.91} \\ 
 		& $m_4$  & 26.36& \textbf{23.01}&23.04&74.15&\textbf{72.99}&73.57&\textbf{89.25}&89.60&90.25  \\
 		& $m_5$  &55.39 &\textbf{21.15}& 21.97&81.90& \textbf{73.65}&74.11&92.39& \textbf{91.44}&92.75 \\
 		& $m_{6}$ &70.25 &\textbf{66.10}&73.51& \textbf{97.77}& 102.39 &118.85 & \textbf{100.62}& 103.24 & 124.56  \\
 		& $m_{7}$ & 71.51 & \textbf{25.66}& 51.12& 87.04 &\textbf{82.02}&105.16& \textbf{95.51}& 100.18& 113.67 \\
 		& $m_{8}$ & \textbf{71.42}&77.46& 101.12 &97.94&\textbf{96.55}&100.96&\textbf{100.17}&102.58&125.90 \\
 		
 		\midrule
 		
 		\multirow{8}{*}{100}
 		& $m_1$  & 3.11 & \textbf{2.95} & 2.97 & 21.89& \textbf{21.83} & 22.00 & \textbf{44.01} & 44.33 & 44.69 \\
 		& $m_2$  & 104.34 & 31.14 & \textbf{31.13} & 103.14 & 81.55 & \textbf{81.46} & 102.43 & 94.64 & \textbf{94.31} \\
 		& $m_3$  & 7.87 & \textbf{3.79}& 4.92&26.36& \textbf{24.02}&24.66& 47.76 &  \textbf{46.49}& 47.61 \\
 		& $m_4$  & 24.67& \textbf{21.92}&22.09&73.73&\textbf{73.48}&73.90&\textbf{88.34}& 88.47& 89.12 \\
 		& $m_5$  &51.18 & \textbf{17.42}&20.54&79.23& \textbf{65.96}& 66.83&90.60&86.14&\textbf{86.09} \\
 		& $m_{6}$ &67.48 &\textbf{61.48}&62.09&\textbf{93.84}&97.15&97.55& \textbf{97.54}&101.00&102.50 \\
 		& $m_{7}$ &64.56 &\textbf{23.62}&167.88& 87.69 &\textbf{77.53}&78.50& 94.60&94.19&\textbf{93.85} \\
 		& $m_{8}$ &\textbf{68.24} &68.44& 69.38 & \textbf{94.66}&98.48& 98.70& \textbf{98.35}&101.23& 102.35 \\
 		
 		\midrule
 		
 		\multirow{8}{*}{300}
 		& $m_1$  & \textbf{3.20} &3.24 &  3.25 & \textbf{20.45} & 22.54 & 22.59 & \textbf{47.83} & 48.33 & 48.41\\
 		& $m_2$  & 101.52 & \textbf{28.85} & 29.07 & 101.95 & \textbf{79.44} & 79.66 & 100.74 & \textbf{94.74} & 94.82 \\
 		& $m_3$  &  7.70 &4.08 & \textbf{3.94}& 27.43 &\textbf{25.55} &25.63 & 49.72 & \textbf{49.57} & 49.81 \\
 		& $m_4$  &26.00 &24.05&\textbf{24.00}&69.44&\textbf{69.32}&69.51&\textbf{87.38}& 87.69 & 87.95 \\
 		& $m_5$  & 56.74& \textbf{18.99}& 19.01& 81.75& \textbf{69.61} &64.96&87.39&86.47&\textbf{86.46} \\
 		& $m_{6}$  & 69.80& 57.92 &\textbf{57.27}& 95.04& 94.68&\textbf{94.60}&\textbf{96.76}&98.60& 99.54 \\
 		& $m_{7}$ &59.92 &20.15&\textbf{20.08}& 83.65& \textbf{71.52}& 71.76& 75.33& \textbf{54.52}&54.8 \\
 		& $m_{8}$ &70.49&66.90& \textbf{67.30}&\textbf{95.08}&96.43&97.00& \textbf{98.89}&99.56&101.32 \\

 		\botrule
 	\end{tabular}
 \end{table}

Across all performance criteria ($\textit{RE}$, $\widehat{MSE}$, and $\textit{RB}$) and the 72 scenarios considered, no single kernel uniformly dominates. The gamma kernel $\mathcal{G}$ accounts for (approximately 30.55\%) of the total number of wins, but this overall result masks a strong dependence on the support of the regression function.

\subsection{Illustrative application}\label{subsec:Illustrations}

We here illustrate the proposed approach on a real dataset using the same criteria as in the simulation study of Section~\ref{subsec:Simulations}. The dataset contains measurements on seeds from three wheat varieties: Kama, Rosa, and Canadian. It comprises $221$ observations on seven geometric variables including surface area, perimeter, compactness index, length, width, asymmetry coefficient, and groove length.\footnote{The data are publicly available from the UCI Machine Learning Repository at \url{https://archive.ics.uci.edu/ml/datasets/seeds}.} We model \texttt{seed length} as a function of \texttt{surface area}. Since both variables are strictly positive, the auxiliary variable is naturally supported on $[0,\infty)$, making the gamma kernel theoretically appropriate for this application.

Table~\ref{descriptive_application} reports some descriptive statistics of the data. Both variables exhibit asymmetric distributions, especially \texttt{seed length}, which presents a marked negative skewness. Such asymmetric configurations are consistent with the situations considered in the simulation study where support-adapted smoothing may provide more stable performances than symmetric kernels. The population of \texttt{seed length}
is plotted in Figure~\ref{fig}. 

The results are reported in Table~\ref{Ratios des MSE_LLR_application}. The gamma kernel achieves the smallest $\widehat{MSE}$ ($0.0071$) and the smallest relative bias ($0.0122$). It also provides the best relative efficiency with $\textit{RE}=0.9184$, indicating an improvement over the Horvitz--Thompson estimator. By contrast, the Gaussian and Epanechnikov kernels both produce relative efficiencies larger than $1$, meaning that they perform worse than the HT estimator in this application. These results are consistent with the simulation study, where the gamma kernel generally showed more stable performances for asymmetric regression functions defined on positive supports.

\begin{table}[!h]
		\caption{Descriptive statistics of \texttt{seed length} (\texttt{y}) and \texttt{surface area} (\texttt{x}) for the seeds dataset ($N=221$).}
		\begin{tabular*}{0.6\textwidth}{c l rrr rrr rrr}
			\toprule
		 & max. & min. & median & mean & sd & CS \\
			\midrule
			 \texttt{y} & 6.675 & 0.818 & 5.554 & 5.564 & 0.701 & $-3.940$ \\
			\midrule
			\texttt{x} &21.18 & 1.00 &14.16 & 14.29& 3.817 & $-0.929$ \\
			\botrule
		\end{tabular*}
		\label{descriptive_application}
\end{table}

\begin{table}[!h]
		\caption{Estimation of the population mean of \texttt{length} using the seed data under simple random sampling ($n=50$), based on $1000$ replications. Mean squared error ($\widehat{MSE}$), relative bias ($\textit{RB}$), and relative efficiency ($\textit{RE}$) with standard deviations in parentheses are reported.}
		\begin{tabular*}{0.85\textwidth}{c l rrr rrr rrr}
			\toprule
			& HT & Gamma & Gaussian & Epanechnikov \\
			\midrule
			$\widehat{MSE} \times 10^{3}$ & 7.73 (9.87) & \textbf{7.10 (10.25)} & 10.49 (33.37) & 13.97 (154.55)\\
			\midrule
			$\textit{RB} \times 10^{2}$ &  & \textbf{1.22 (0.89)} & 1.26 (1.34) & 1.34 (1.64)\\
			\midrule
			\textit{RE}  (\%) & 1 & \textbf{91.84} & 135.68 & 180.66 \\
			\botrule
		\end{tabular*}
		\label{Ratios des MSE_LLR_application}
\end{table}

\begin{figure}[bt]
\centering
\includegraphics[width=10cm]{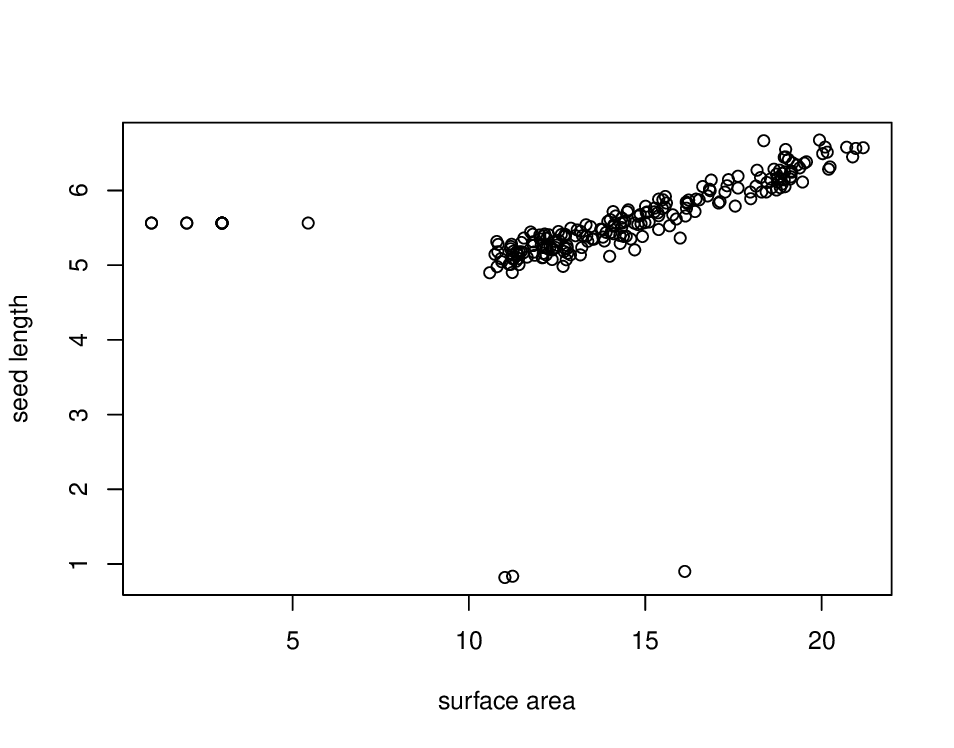}
\caption{Plot of the \texttt{seed length} population.}
\label{fig}
\end{figure}

\newpage
\section{Conclusion}\label{sec:Conclusion}
 
 This paper proposed a general model-assisted estimator of the finite population mean based on local linear nonparametric regression with associated kernels, a class that encompasses both classical symmetric
 kernels and asymmetric kernels. The asymptotic properties of the proposed estimator including design-unbiasedness, consistency, mean squared
 error, asymptotic normality, and attainment of the GJ lower bound were derived theoretically under a unified framework. The finite-sample performance of the estimator was evaluated through
 simulation studies and one real-data application, where comparisons were conducted with estimators based on the Gaussian and Epanechnikov kernels
 and with the HT estimator, focusing on the gamma kernel as a representative asymmetric member of the associated class.
 
 The results show that the gamma kernel generally improves estimation accuracy relative to symmetric kernels when the auxiliary variable is supported on a bounded or semi-infinite domain, with the advantage most
 pronounced at medium to large sample sizes and for regression functions with monotone or asymmetric structure. The competitive performance of symmetric kernels in several scenarios, in particular for the quadratic function symmetric about an interior point, should not be read as evidence of their general suitability for positive data: it primarily
 reflects the masking of boundary bias by sampling variance at small sample sizes. In practice, kernel choice should be guided by the support
 structure of the auxiliary variable, and the gamma kernel is the natural default when that support is the positive real line.
 
 Several directions for future research are worth pursuing. First, the proposed framework could be extended to stratified or cluster sampling
 designs, where the interplay between the sampling structure and the kernel smoother raises additional theoretical questions. Second, other
 asymmetric associated kernels, such as the beta kernel for data supported
 on a compact interval, could be investigated within the model-assisted
 framework, following the general asymptotic results established here.
 Third, a semi-parametric extension, in which some covariates enter the model parametrically and others nonparametrically, would broaden the
 applicability of the approach. Finally, the multivariate setting involving several auxiliary variables represents a natural and
 practically important generalization.
  \section{Appendix}\label{sec:Annexe}	

\begin{proof}[Proof of Theorem~\ref{thrm}]
	
	We work conditionally on the design points
	$x = (x_1,\ldots,x_n)$ throughout; all expectations,
	variances, and covariances are conditional on $x$ unless
	otherwise stated.
	
	We record two consequences used repeatedly below. 	
	First, by the variance decomposition and Assumption~\ref{A6}:
	\begin{equation}
		A_2(x_j,h) = B(x_j,h) + {A(x_j,h)}^2 = B(x_j,h)\{1+o(1) \}.
		\label{A2_decomp}
	\end{equation}
	
	Second, since $A_k(x_j,h) \to 0$ for all $k = 1,2$ as
	$h \to 0$ by ~\eqref{definition:associated_kernel},
	while Assumptions~\ref{A4}--\ref{A5} together imply
	$\|K_{x_j,h}\|_2^2 \to \infty$,
	we have, from \eqref{Akxh}:
	\begin{equation}
		\frac{n}{\pi_1^2}\,A_{l+1}(x_j,h)
		= o\!\left(\frac{n}{\pi_1^2}\,\|K_{x_j,h}\|_2^2\right)
		\quad\text{for all } l =0,1.
		\label{moments_vs_norm}
	\end{equation}

Then, the proof is done in five steps.
	
	
	\paragraph{Step 1: Bias decomposition}
	
	According to \citet[Theorem~1]{chen2002local}, the
	conditional expectation of
	$\widehat{\beta} = (\widehat{\beta}_0,\widehat{\beta}_1)^\top$
	satisfies:
	\[
	\mathbb{E}(\widehat{\beta} \mid x)
	= \left(\mathcal{X}_s^\top\mathcal{W}_s\mathcal{X}_s\right)^{-1}
	\mathcal{X}_s^\top\mathcal{W}_s\,\mathbb{E}(Y_s \mid x).
	\]
	Setting $\mathbb{E}(Y_s \mid x)
	= m = (m(x_1),\ldots,m(x_n))^\top$ and
	$\beta = (m(x_j),m^{(1)}(x_j))^\top$, we write
	$m = \mathcal{X}_s\beta + r$ where
	$r = m - \mathcal{X}_s\beta$ is the local polynomial
	residual. Then, one has:
	\[
	\operatorname{Bias}(\widehat{\beta} \mid x)
	= \left(\mathcal{X}_s^\top\mathcal{W}_s\mathcal{X}_s
	\right)^{-1}
	\mathcal{X}_s^\top\mathcal{W}_s\,r.
	\]
	
	
	\paragraph{Step 2: Asymptotic moments of $S_{nl}$}
	
	The weighted design matrix satisfies:
	\[
	\mathcal{X}_s^\top\mathcal{W}_s\mathcal{X}_s
	= \begin{pmatrix} S_{n0} & S_{n1} \\ S_{n1} & S_{n2}
	\end{pmatrix},
	\qquad
	S_{nl} = \sum_{i=1}^n \frac{1}{\pi_i}(x_i-x_j)^l
	K_{x_j,h}(x_i).
	\]
	Under Assumption~\ref{A1}, $\pi_i = n/N$ for all $i \in s$.
	Since the $x_i$ have density $f$ by Assumption~\ref{A2},
one successively has:
	\begin{align}
		\mathbb{E}(S_{nl})&=\frac{n}{\pi_1}
		\mathbb{E}\!\left\{(x_1-x_j)^l K_{x_j,h}(x_1)\right\}
		\nonumber\\ 
		&= \frac{n}{\pi_1}
		\int_{\mathbb{S}_{x_j,h}}(s-x_j)^l
		K_{x_j,h}(s)\,f(s)\,ds \nonumber \\
		&= \frac{n}{\pi_1}
		\mathbb{E}\!\left\{
		(\mathcal{Z}_{x_j,h}-x_j)^l f(\mathcal{Z}_{x_j,h})
		\right\},
		\label{E(Snl)}
	\end{align}
	where $x_j \in \mathbb{S}_{x_j,h}$ \eqref{Akxh}.

	To evaluate \eqref{E(Snl)} with  \eqref{definition:associated_kernel}, we apply the Taylor series expansion of $f(\mathcal{Z}_{x_j,h})$ around $\mathbb{E}(\mathcal{Z}_{x_j,h}) = x_j + A(x_j,h)$: \begin{align}
		f(\mathcal{Z}_{x_{j}, h}) & = f  \left \{x_{j} + A (x_{j}, h) \right \} + \left\{  \mathcal{Z}_{x_{j}, h} - x_{j} - A (x_{j}, h)  \right\} f^{\left(1\right)} \left \{ x_{j} + A (x_{j}, h)\right \} \nonumber\\
		& \  \  \  + \frac{1}{2} \left\{  \mathcal{Z}_{x_{j}, h} - x_{j} - A (x_{j}, h )  \right\}^{2} f^{\left(2 \right)} \left \{ x_{j} + A (x_{j}, h) \right \}  \label{e} \\
		& \  \  \  + o \left ( \left [  \mathcal{Z}_{x_{j}, h} - \left \{ x_{j} + A(x_{j}, h) \right \} \right ]^{2} \right ). \nonumber
	\end{align}
	
	Then using another expansion around $x_j$
	\begin{equation}
		\mathbb{E}\!\left\{
		(\mathcal{Z}_{x_j,h}-x_j)^l f(\mathcal{Z}_{x_j,h})
		\right\}
		= \sum_{k=0}^{2} f^{(k)}(x_j)\,\frac{A_{l+k}(x_j,h)}{k!}
		+ o\left\{A_{l+2}(x_j,h)\right\}.
		\label{expansion_Snl}
	\end{equation}
	Substituting into \eqref{E(Snl)} for $l = 0, 1, 2$
	and using \eqref{A2_decomp} and Assumption~\ref{A6}:
	\begin{align*}
		\mathbb{E}(S_{n0})
		&= \frac{n}{\pi_1}f(x_j)\{1+o(1)\}, \\
		\mathbb{E}(S_{n2})
		&= \frac{n}{\pi_1}f(x_j)A_2(x_j,h)\{1+o(1)\}, \\
		\mathbb{E}(S_{n1})
		&= \frac{n}{\pi_1}f(x_j)A(x_j,h)\{1+o(1)\}
		= o\!\left\{ \mathbb{E}(S_{n2})^{1/2}\right\}.
	\end{align*}
	Thus $S_{n1}$ is asymptotically negligible relative to the
	diagonal terms, and the expected design matrix satisfies:
	\[
	\mathbb{E}\!\left(
	\mathcal{X}_s^\top\mathcal{W}_s\mathcal{X}_s
	\right)
	\approx \frac{n}{\pi_1}f(x_j)
	\begin{pmatrix} 1 & 0 \\ 0 & A_2(x_j,h) \end{pmatrix}
	\{1+o(1)\},
	\] and 	\begin{equation}
		\left(\mathbb{E}\!\left(
		\mathcal{X}_s^\top\mathcal{W}_s\mathcal{X}_s
		\right)\right)^{-1}
		\approx \frac{\pi_1}{nf(x_j)}
		\begin{pmatrix} 1 & 0 \\ 0 & A_2(x_j,h)^{-1}
		\end{pmatrix}\{1+o(1)\},
		\label{(X'WX)}
	\end{equation}
	using $f(x_j) \geq f_c > 0$ of Assumption~\ref{A3} and
	$A_2(x_j,h) > 0.$

	
	\paragraph{Step 3: Residual and leading bias term}
	
	By Assumption~\ref{A2}, $m^{(2)}$ is continuous on
	$\mathbb{S}_{x_j,h}$, so a Taylor expansion of $m$
	around $x_j$ gives:
	\[
	m(x_i)
	= m(x_j) + m^{(1)}(x_j)(x_i-x_j)
	+ \tfrac{1}{2}m^{(2)}(x_j)(x_i-x_j)^2
	+ o\!\left\{(x_i-x_j)^2\right\}.
	\]
	Since $\mathcal{X}_s\beta
	= m(x_j)\mathbf{1} + m^{(1)}(x_j)(x_i-x_j)$ by
	definition with $\mathbf{1}=(1,\ldots,1)^{\top}$, the residual satisfies
	$r_i = \tfrac{1}{2}m^{(2)}(x_j)(x_i-x_j)^2
	+ o((x_i-x_j)^2)$.
	Substituting and recognising $S_{n2}$, $S_{n3}$
	via \eqref{E(Snl)}:
	\[
	\mathcal{X}_s^\top\mathcal{W}_s r
	= \frac{m^{(2)}(x_j)}{2}
	\begin{pmatrix} S_{n2} \\ S_{n3} \end{pmatrix}\{1+o(1) \}
	= \frac{m^{(2)}(x_j)}{2}\,\frac{nf(x_j)}{\pi_1}
	\begin{pmatrix} A_2(x_j,h) \\ A_3(x_j,h)
	\end{pmatrix} \{ 1+o(1) \}.
	\]
	Combining with \eqref{(X'WX)} yields an approximation of the mean bias:
	\begin{equation}
		e_1^\top
		\!\left(\mathbb{E}\!\left(
		\mathcal{X}_s^\top\mathcal{W}_s\mathcal{X}_s
		\right)\right)^{-1}
		\!\mathbb{E}\!\left(
		\mathcal{X}_s^\top\mathcal{W}_s r
		\right)
		= \frac{1}{2}\,m^{(2)}(x_j)\,A_2(x_j,h).
		\label{leading_bias}
	\end{equation}


\paragraph{Step 4: Stochastic remainder}

From Step~3, the bias satisfies:
\[
\operatorname{Bias}(\widehat{m}_j \mid x)
= \frac{m^{(2)}(x_j)}{2}
\cdot e_1^\top
\left(\mathcal{X}_s^\top\mathcal{W}_s\mathcal{X}_s
\right)^{-1}
\begin{pmatrix} S_{n2} \\ S_{n3} \end{pmatrix}
\{1+o(1) \}.
\]
Replacing $(\mathcal{X}_s^\top\mathcal{W}_s\mathcal{X}_s)^{-1}$
by its leading approximation \eqref{(X'WX)}, we obtain:
\begin{equation}
	\operatorname{Bias}(\widehat{m}_j \mid x)
	= \frac{m^{(2)}(x_j)}{2}
	\cdot \frac{S_{n2}}{S_{n0}}
	\cdot\{1+o(1) \}.
	\label{bias_ratio}
\end{equation}
It therefore suffices to expand the  ratio
$S_{n2}/S_{n0}$ around its deterministic counterpart
$\mathbb{E}(S_{n2})/\mathbb{E}(S_{n0})$.

\medskip
\noindent\textit{Bounding $\operatorname{Var}(S_{nl})$.}
One has $S_{nl} = \mathbb{E}(S_{nl}) + \Delta_{nl}$ with
$\Delta_{nl} = S_{nl} - \mathbb{E}(S_{nl})$.
Under Assumption~\ref{A1}:
\[
\operatorname{Var}(S_{nl})
= \frac{n}{\pi_1^2}\!\left\{
\int_{\mathcal{I}}(t-x_j)^{2l}
K_{x_j,h}^2(t)\,f(t)\,dt
- \left(
\int_{\mathcal{I}}(t-x_j)^{l}
K_{x_j,h}(t)\,f(t)\,dt
\right)^{\!2}
\right\}.
\]
From Assumption~\ref{A2},  $f(t) = f(x_j) + R(t)$ with
$|R(t)| \leq C|t-x_j|$ 
and denoting the two integrals $I_1 - I_2$ gives:

\begin{itemize}
	\item 
	By \citet[Theorem~1]{chen2002local}:
	\(
	\int_{\mathcal{I}}(t-x_j)^{2l}K_{x_j,h}^2(t)\,dt
	= \mathcal{O}(\|K_{x_j,h}\|_2^2),
	\)
	so $I_1 = \tfrac{n}{\pi_1^2}f(x_j)\,
	\mathcal{O}(\|K_{x_j,h}\|_2^2)\{1+o(1)\}$;
	
	\item and the rest $I_2$ is

	\[
	\frac{n}{\pi_1^2}
	\left|
	\int_{\mathcal{I}}(t-x_j)^{l}K_{x_j,h}(t)\,R(t)\,dt
	\right|
	\leq \frac{Cn}{\pi_1^2}\,A_{l+1}(x_j,h)
	= o\!\left(
	\frac{n}{\pi_1^2}\,\|K_{x_j,h}\|_2^2
	\right),
	\]
	using \eqref{moments_vs_norm} and since $|R(t)| \leq C|t-x_j|$.
\end{itemize}
 Thus, using $\pi_1 = n/N$ leads to:
\begin{equation}
	\operatorname{Var}(S_{nl})
	= \mathcal{O}\!\left(\frac{\|K_{x_j,h}\|_2^2}{n}\right),
	\quad \text{and} \quad
	\Delta_{nl}
	= \mathcal{O}_p\!\left\{
	\left(\frac{\|K_{x_j,h}\|_2^2}{n}
	\right)^{\!1/2}
	\right\}.
	\label{var_bound}
\end{equation}

\medskip
\noindent\textit{ Expansion of the ratio
	$S_{n2}/S_{n0}$.} A first-order Taylor expansion of $S_{n2}/S_{n0}$
around $(\mathbb{E}(S_{n2}), \mathbb{E}(S_{n0}))$
gives:
\[
\frac{S_{n2}}{S_{n0}}
= \frac{\mathbb{E}(S_{n2})}{\mathbb{E}(S_{n0})}
+ \frac{\Delta_{n2}}{\mathbb{E}(S_{n0})}
- \frac{\mathbb{E}(S_{n2})}{\mathbb{E}(S_{n0})^2}
\,\Delta_{n0}
+ R_n,
\]
where  $R_n$ is a second order remainder involving products
$\Delta_{nl_1}\Delta_{nl_2}$ and satifies
$o_p(\Delta_{n0}^2 + \Delta_{n2}^2)$.
From Step~2, $\mathbb{E}(S_{n0})
= \tfrac{n}{\pi_1}f(x_j)\{1+o(1)\} = Nf(x_j)\{1+o(1)\}$,
so the two fluctuation terms satisfy:
\[
\frac{\Delta_{n2}}{\mathbb{E}(S_{n0})}
= \mathcal{O}_p\!\left\{
\frac{1}{Nf(x_j)}
\left(\frac{\|K_{x_j,h}\|_2^2}{n}\right)^{\!1/2}
\right\},
\ \ 
\frac{\mathbb{E}(S_{n2})}{\mathbb{E}(S_{n0})^2}
\,\Delta_{n0}
= \mathcal{O}_p\!\left\{
\frac{A_2(x_j,h)}{Nf(x_j)}
\left(\frac{\|K_{x_j,h}\|_2^2}{n}\right)^{\!1/2}
\right\},
\]
using \eqref{var_bound} and the ratio
$\mathbb{E}(S_{n2})/\mathbb{E}(S_{n0})
= A_2(x_j,h)\{1+o(1)\}$ from Step~2.
Both terms are therefore of order:
\begin{equation}
	\mathcal{O}_p\!\left(
	\frac{\|K_{x_j,h}\|_2^2}{n}
	\right)
	= \mathcal{O}_p\!\left(\frac{1}{nh^{r_2}}\right),
	\label{fluct_bound}
\end{equation}
using Assumption~\ref{A4}.

\medskip

Substituting the deterministic ratio from Step~2,
\[
\frac{\mathbb{E}(S_{n2})}{\mathbb{E}(S_{n0})}
= A_2(x_j,h)\{1+o(1)\},
\]
and  \eqref{fluct_bound}
into \eqref{bias_ratio}:
\[
\operatorname{Bias}\!\left(\widehat{m}_j \mid x\right)
= \frac{1}{2}\,m^{(2)}(x_j)\,A_2(x_j,h)
+ o\!\left\{A_2(x_j,h)\right\}
+ \mathcal{O}\!\left(\frac{1}{nh^{r_2}}\right),
\]
which establishes~\eqref{Bias_m}.
	
	
	\paragraph{Step 5: Conditional variance}
	
	The conditional variance of $\widehat{\beta}$ follows from:
	\[
	\operatorname{Var}\!\left(\widehat{\beta} \mid x\right)
	= \left(\mathcal{X}_s^\top\mathcal{W}_s\mathcal{X}_s
	\right)^{-1}
	\mathcal{X}_s^\top\mathcal{W}_s^2\,
	\operatorname{Var}(Y_s \mid x)\,
	\mathcal{X}_s
	\left(\mathcal{X}_s^\top\mathcal{W}_s\mathcal{X}_s
	\right)^{-1},
	\]
	where $\mathcal{W}_s^2
	= \operatorname{diag}\{\pi_i^{-2}K_{x_j,h}^2(x_i)\}$
	and $\operatorname{Var}(y_i \mid x_i) = \sigma^2(x_i)$.
	Defining $\Psi_s
	= \operatorname{diag}\{\pi_i^{-2}K_{x_j,h}^2(x_i)
	\sigma^2(x_i)\}$:
	\[
	\operatorname{Var}\!\left(\widehat{m}_j \mid x\right)
	= e_1^\top
	\left(\mathcal{X}_s^\top\mathcal{W}_s\mathcal{X}_s
	\right)^{-1}
	\mathcal{X}_s^\top\Psi_s\mathcal{X}_s
	\left(\mathcal{X}_s^\top\mathcal{W}_s\mathcal{X}_s
	\right)^{-1}
	e_1.
	\]
	Set $\mathcal{X}_s^\top\Psi_s\mathcal{X}_s
	= \bigl(\begin{smallmatrix}
		T_{n0} & T_{n1} \\ T_{n1} & T_{n2}
	\end{smallmatrix}\bigr)$
	with $T_{nl} = \sum_{i=1}^n \pi_i^{-2}(x_i-x_j)^l
	K_{x_j,h}^2(x_i)\sigma^2(x_i)$.
	By Assumption~\ref{A2}, $\sigma^2$ is Lipschitz on
	$\mathbb{S}_{x_j,h}$, so
	$|\sigma^2(x_i) - \sigma^2(x_j)| \leq L|x_i-x_j|$
	for some $L < \infty$. Replacing $\sigma^2(x_i)$ by
	$\sigma^2(x_j)$ introduces an error bounded by
	$L\,A_1(x_j,h) \to 0$, which is of smaller order than the
	leading term. Applying the same argument
	as in \eqref{E(Snl)} with $K_{x_j,h}^2$:
	\[
	\mathbb{E}(T_{nl})
	= \frac{n}{\pi_1^2}\,\sigma^2(x_j)\,f(x_j)\,
	\nu_l(x_j)\{1+o(1)\},
	\qquad
	\nu_l(x_j)
	= \int_{\mathbb{S}_{x_j,h}}(t-x_j)^l
	K_{x_j,h}^2(t)\,dt,
	\]
	with $\nu_0(x_j) = \|K_{x_j,h}\|_2^2$ in particular. 
	As for $S_{n1}$, using
	$A_1(x_j,h) \to 0$ and $\|K_{x_j,h}\|_2^2 \to \infty$
	from \eqref{moments_vs_norm}, the off-diagonal term $T_{n1}$
	is asymptotically negligible, giving:
	\[
	\mathbb{E}\!\left(
	\mathcal{X}_s^\top\Psi_s\mathcal{X}_s
	\right)
	\approx \frac{n}{\pi_1^2}\,\sigma^2(x_j)\,f(x_j)\,
	S^*\{1+o(1)\},
	\qquad
	S^* = \begin{pmatrix}
		\nu_0(x_j) & \nu_1(x_j) \\
		\nu_1(x_j) & \nu_2(x_j)
	\end{pmatrix}.
	\]
	Substituting \eqref{(X'WX)} and setting
	$D = \operatorname{diag}(1,A_2(x_j,h)^{-1})$:
	\begin{align*}
		\operatorname{Var}(\widehat{m}_j \mid x)
		&\approx
		\frac{\pi_1}{nf(x_j)}\,D
		\cdot\frac{n}{\pi_1^2}\,\sigma^2(x_j)f(x_j)\,S^*
		\cdot\frac{\pi_1}{nf(x_j)}\,D\,\{1+o(1)\} \\[4pt]
		&= \frac{\sigma^2(x_j)}{nf(x_j)}\,
		[D\,S^*\,D]\,\{1+o(1)\}.
	\end{align*}
	Since $e_1^\top M e_1 = M_{11}$ for any matrix $M$,
	only the $(1,1)$ entry contributes. Computing explicitly:
	\[
	[D\,S^*\,D]{11} = \nu_0(x_j) = \|K_{x_j,h}\|_2^2.
	\]
	Therefore:
	\[
	\operatorname{Var}\!\left(\widehat{m}_j \mid x\right)
	= \frac{\sigma^2(x_j)}{nf(x_j)}\,\|K_{x_j,h}\|_2^2
	+ o\!\left(\frac{\|K_{x_j,h}\|_2^2}{n}\right)
	= \frac{\sigma^2(x_j)}{nf(x_j)}\,\|K_{x_j,h}\|_2^2
	+ o\!\left(\frac{1}{nh^{r_2}}\right),
	\]
	where the last equality uses Assumption~\ref{A4}.
	This establishes~\eqref{Var_m} and completes the proof.
\end{proof}

\begin{lemma}\label{Lemme6}
Under Assumptions~\textnormal{\ref{B1}-\ref{B5}}, 
\[ \lim_{N \longrightarrow \infty}^{} \frac{1}{N} \sum_{j \in U}^{} \mathbb{E} \left( m_{j}-m(x_{j}) \right)^{2} = 0. \]
\end{lemma}
\begin{proof}Following \citet[Lemma 6]{breidt2000local}, by Theorem \ref{thrm}, this follows directly from standard local polynomial regression theory; see \citet[Theorem 1]{chen2002local}.
\end{proof}

\begin{lemma}\label{Lemme2}
Assume \textnormal{\ref{B1}-\ref{B6}}:
\begin{enumerate}
\item [$\left(\textnormal{i}\right)$]  For $k \geq 0,\ \lim_{N \longrightarrow \infty}^{} \sup \frac{1}{N} \sum_{j \in U}^{} \left( \frac{1}{2Nh} \sum_{i \in U}^{} \mathds{1}_{ \left\{x_{j}-h \leq x_{i} \leq x_{j}+h \right\}} \right)^{k} < \infty.$
\item  [$\left(\textnormal{ii}\right)$] $\exists\ N^{\ast}\ \text{tel que}\ N \geq N^{\ast} \Rightarrow \sum_{k \in U}^{} \mathds{1}_{\left\{ \left| x-x_{k} \right| \leq h \right\}} \geq 2.$
\item  [$\left(\textnormal{iii}\right)$] Les $\overline{y}_{jg},\ m_{j}$ are uniformly bounded at $j$ and the $\widehat{\overline{y}}_{jg},\ \widehat{m}_j$ in $j$ and $s$.
\item [$\left(\textnormal{iv}\right)$] The first four partial derivatives of $\widehat{m}_j$ by considering $\overline{y}_{jg}\ \text{and}\ \delta,$ evaluated $\text{at}\ \widehat{\overline{y}}_{j} = \overline{y}_{j},\ \delta = 0$, are uniformly bounded at $j.$
\item [$\left(\textnormal{v}\right)$] Les $R^{2}_{jN}$ are uniformly bounded at $j$ and $s.$ 
\end{enumerate}
\end{lemma}
\begin{proof}
For $\left(\textnormal{i}\right),\ \left(\textnormal{ii}\right),$ see \citet[Lemma 2 $\left(\textnormal{i}\right),\ \left(\textnormal{ii}\right)$]{breidt2000local}.

For $\left(\textnormal{iii}\right),\ \left(\textnormal{iv}\right),\ \left(\textnormal{v}\right),$ by assuming \textnormal{\ref{B1}-\ref{B6}},
\begin{align*}
\lim_{N \longrightarrow \infty}^{} \sup \left| \widehat{\overline{y}}_{jg} \right| & = \lim_{N \longrightarrow \infty}^{} \sup \left| \sum_{k \in U}^{}\frac{1}{N} K_{x_{j},h}\left(x_{k}\right) (x_{k}-x_{j})^{l_{1}}{y_{k}}^{l_{2}} \frac{\mathds{1}_k}{\pi_k} \right| \\
&  \leq \lim_{N \longrightarrow \infty}^{} \sup \sum_{k \in U}^{}\frac{c}{Nh \lambda} \mathds{1}_{\left\{ x_{j}-h \leq x_{k} \leq x_{j}+h \right\}},
\end{align*}
is uniformly bounded at $j$ as well as $\overline{y}_{jg},\ \text{for}\ \lambda < 1$, following \citet[Lemma 2 $\left(\textnormal{iii} \right)$]{breidt2000local}.
Similarly, $m_j,\ \widehat{m}^{1}_j$ and their derivatives are respectively continuous functions of $\overline{y}_{jg},\ \widehat{\overline{y}}_{jg}$, with non-zero denominators. Using \eqref{Rj}, the quantities $R^{2}_{jN}$ are uniformly bounded in $ j $ and $ s $. 
\end{proof}

\begin{lemma}\label{Lemme3}
Assuming \textnormal{\ref{B1}-\ref{B7}}, according to \eqref{Rj},
\[ \frac{n}{N} \sum_{i \in U}^{} \mathbb{E}_{p} \left( R^{2}_{jN} \right) = O \left( \frac{1}{nh^{2}} \right). \]
\end{lemma}

\begin{proof} The proof follows the same procedure as in \citet[Lemma 3]{breidt2000local}, using the associated kernels instead of the classical kernels. Let $Q_{N} = \frac{1}{N} \sum_{j \in U}^{} \mathbb{E}_{p} \left| \widehat{\overline{y}}_{jg} - \overline{y}_{jg} \right|^{4}.$ Note that
\begin{align*}
Q_{N} & = \frac{1}{N} \sum_{j \in U}^{} \mathbb{E}_{p} \left| \frac{1}{N} \sum_{k \in U}^{} K_{x_{j},h} \left(x_{k}\right) z^{+}_{jgk} \left(1-\frac{\mathds{1}_k}{\pi_k} \right) \right|^{4}\\
& = \frac{1}{N^{5}}  \mathbb{E}_{p} \left| \sum_{j,k,l,m,n \in U}^{} K_{x_{j},h} \left(x_{k}\right) z^{ +}_{jgk} \left(1-\frac{\mathds{1}_k}{\pi_k} \right) \times K_{x_{j},h} \left(x_{l}\right) z^{ +}_{jgl} \left(1-\frac{\mathds{1}_l}{\pi_l} \right) \right.\\
& \  \  \  \times \left. K_{x_{j},h} \left(x_{m}\right) z^{ +}_{jgm} \left(1-\frac{\mathds{1}_m}{\pi_m} \right) \times K_{x_{j},h} \left(x_{n}\right) z^{ +}_{jgn} \left(1-\frac{\mathds{1}_n}{\pi_n} \right) \right| \\
& \leq \frac{c}{N^{5} h^{4}} \sum_{j, k, l, m, n \in U}^{} \mathds{1}_{\left\{x_{j}-h \leq x_{k},x_{l},x_{m},x_{n} \leq x_{j}+h \right\}} \\
& \ \ \  \times  \mathbb{E}_{p}\left| \left\{ \left(\mathds{1}_{k} - \pi_{k} \right)\left(\mathds{1}_{l} - \pi_{l} \right)\left(\mathds{1}_{m} - \pi_{m} \right) \left(\mathds{1}_{n} - \pi_{n} \right) \right\} \right|, 
\end{align*}
which remains bounded by \textnormal{\ref{B5}}, \textnormal{\ref{B7}} and Lemma \ref{Lemme2} $\left( i \right)$. Based on the same arguments as \citet[Lemma 3]{breidt2000local}, we recover the result. 
\end{proof}

\begin{lemma}\label{Lemme4}
Under Assumptions~\textnormal{\ref{B1}-\ref{B7}}, then \[ \lim_{N\longrightarrow \infty}^{} \frac{1}{N} \sum_{j \in U} \mathbb{E}_{p} \left( \widehat{m}^{1}_{j} - m_{j}\right)^{2} = 0. \]
\end{lemma}

\begin{proof}
 Following \citet[Lemma 4]{breidt2000local}, by \eqref{Rj}, 
\begin{align*}
 \widehat{m}^{1}_{j} - m_{j} & = R_{jN}+ \frac{1}{N} \sum_{k \in U}^{} \left\{ \sum_{g = 1}^{5}  \frac{\partial \widehat{m}^{1}_{j}}{\partial \left(\widehat{ \overline{y}}_{jg}\right)} \bigg|_{\widehat{ \overline{y}}_{j} =  \overline{y}_{j}, \delta = 0 }K_{x_{j},h} \left(x_{k}\right) z^{+}_{jgk} \right\} \left(\frac{\mathds{1}_k}{\pi_k} - 1\right) + o \left( \frac{\delta}{N^{2}} \right)\\
 & = R_{jN}+ \frac{1}{N} \sum_{k \in U}^{} z_{jk} \left(\frac{\mathds{1}_k}{\pi_k} - 1\right) + o \left( \frac{\delta}{N^{2}} \right),
\end{align*} then
 \begin{align} \label{m^j-mj}
 \frac{1}{N} \sum_{j \in U}^{} \mathbb{E}_{p}  \left(  \widehat{m}^{1}_{j} - m_{j} \right)^{2} & = \frac{1}{N} \sum_{j \in U}^{} \mathbb{E}_{p} \left( R^{2}_{jN} \right) + \frac{2}{N^{2}} \sum_{j \in U}^{} \sum_{k \in U}^{} z_{jk} \mathbb{E}_{p} \left\{ \left( \frac{\mathds{1}_k}{\pi_k} - 1 \right) R_{jN} \right\} \\
 & \  \  \  \ + \frac{1}{N^{3}} \sum_{j \in U}^{} \mathbb{E}_{p} \left[  \left\{  \sum_{k \in U}^{} z_{jk} \left( \frac{\mathds{1}_k}{\pi_k} - 1 \right) \right\}^{2} \right] + o \left( \frac{\delta}{N^{2}} \right), \nonumber
 \end{align} where 
 \begin{align*}
  \frac{1}{N^{3}} \sum_{j \in U}^{} \mathbb{E}_{p} \left[  \left\{  \sum_{k \in U}^{} z_{jk} \left( \frac{\mathds{1}_k}{\pi_k} - 1 \right) \right\}^{2} \right]  & =  \frac{1}{N^{3}} \sum_{j \in U}^{}\sum_{k \in U}^{} z^{2}_{jk} \frac{1-\pi_{k}}{\pi_{k}} \\
  & \ \ \ \ \ + \frac{1}{N^{3}} \sum_{j \in U}^{}\sum_{{k,l \in U}_{k \neq l}}^{} z_{jk} z_{jl} \frac{\pi_{kl}-\pi_{k}\pi_{l}}{\pi_{k} \pi_{l}} \\
  &  \leq  \frac{1}{\lambda N^{3}} \sum_{j,k \in U}^{} z^{2}_{jk}\\
  & \ \ \ \ \  + \frac{N \max_{{i,j \in U}_{i \neq j}}^{} \left| \pi_{ij}-\pi_{i}\pi_{j} \right|}{\lambda^{2} N^{3}}\sum_{j,k \in U}^{} z^{2}_{jk}
 \end{align*}
tend to $0$ by \textnormal{\ref{B6}}, and following \citet[ Lemma 1]{breidt2000local} and Lemma \ref{Lemme2} $\left( v \right).$ The first term in \eqref{m^j-mj} converges to $0$ when considering Lemma \ref{Lemme3}. The second term converges to $0$ by the Cauchy-Schwarz inequality. 
\end{proof}

\begin{lemma}\label{Lemme5}
Under Assumptions~\textnormal{\ref{B1}-\ref{B7}}, then
\begin{equation*}
\lim_{N \longrightarrow \infty}^{} \frac{n}{N^{2}} \mathbb{E}_{p} \left[ \sum_{i,j \in U}^{} \left( \widehat{m}^{1}_{i}-m_{i} \right) \left( \widehat{m}^{1}_{j}-m_{j} \right) \left( 1-\frac{\mathds{1}_i}{\pi_i} \right) \left( 1-\frac{\mathds{1}_j}{\pi_j} \right) \right] = 0.
\end{equation*}
\end{lemma}

\begin{proof} Following \citet[Lemma 5]{breidt2000local}, by  \eqref{Rj},
\begin{align*}
  &  \frac{n}{N^{2}} \mathbb{E}_{p} \left[ \sum_{i,j \in U}^{} \left( \widehat{m}^{1}_{i}-m_{i} \right) \left( \widehat{m}^{1}_{j}-m_{j} \right) \left( 1- \frac{\mathds{1}_i}{\pi_i} \right) \left( 1-\frac{\mathds{1}_j}{\pi_j} \right) \right] \\
  & = \frac{n}{N^{2}} \sum_{i,j \in U}^{} \mathbb{E}_{p} \left\{ R_{iN}R_{jN} \left( 1-\frac{\mathds{1}_i}{\pi_i} \right) \left( 1-\frac{\mathds{1}_j}{\pi_j} \right) \right\} \\
 & \  \  \  \ + \frac{2n}{N^{3}}  \sum_{i,j,k \in U}^{} z_{jk} \mathbb{E}_{p} \left\{ R_{jN} \left( 1-\frac{\mathds{1}_i}{\pi_i} \right) \left( 1-\frac{\mathds{1}_j}{\pi_j} \right)   \left( 1-\frac{\mathds{1}_k}{\pi_k} \right) \right\} \\
&  \ \ \  \ + \frac{n}{N^{4}}  \sum_{i,j,k,l \in U}^{} z_{jk}z_{jl} \mathbb{E}_{p} \left\{ \left( 1-\frac{\mathds{1}_i}{\pi_i} \right) \left( 1-\frac{\mathds{1}_j}{\pi_j} \right) \left( 1-\frac{\mathds{1}_k}{\pi_k} \right)\left( 1-\frac{\mathds{1}_l}{\pi_l} \right) \right\} + o\left(1 \right) \\
  & = q_{1N}+q_{2N}+q_{3N}+ o\left(1 \right).
\end{align*}
The term $q_{1N}$ converges to $0$ by Lemma \ref{Lemme3} and \textnormal{\ref{B6}}. The term $q_{2N}$ converges to $0$ by the Cauchy-Schwarz inequality. By separating the elements for $\left( i,\ j,\ k,\ l \right)$ as in Lemma \ref{Lemme3}, $q_{3N}$ also converges to 0, and thus the result is proved.
\end{proof}


\end{document}